\documentclass[10pt,a4paper]{elsarticle} 
\usepackage{multicol}
\usepackage{flushend}
\usepackage{longtable}
\usepackage{threeparttable}
\usepackage[bookmarks=true]{hyperref}%\usepackage[dvips=true,bookmarks=true]{hyperref}
\usepackage{times}
\usepackage{color}
\usepackage{setspace}
\usepackage{balance}
\usepackage{indentfirst}
\usepackage{lastpage}
\usepackage{hyperref}
\usepackage[tbtags]{amsmath}
\usepackage{nomencl}
\usepackage{pstricks}
\usepackage{pstricks-add}
\usepackage{pst-plot,pstricks-add}

\usepackage{enumerate}
\usepackage{amsmath,amsthm}
\usepackage{amsfonts}
\usepackage{amssymb}
\usepackage{algorithm, algorithmic}
\usepackage{hyperref}
\usepackage{xcolor}
\usepackage{graphicx}
\usepackage{arydshln}
\usepackage{subcaption}
\usepackage{tikz}
\usetikzlibrary{matrix,positioning,decorations.pathreplacing}
\usepackage[width=15cm, left=2cm, right=2cm, vscale=0.8, nohead]{geometry} 
\numberwithin{equation}{section}
\newtheorem{theorem}{Theorem}[section]
\newtheorem{proposition}[theorem]{Proposition}
\newtheorem{lemma}[theorem]{Lemma}
\newtheorem{corollary}[theorem]{Corollary}
\theoremstyle{definition}
\newtheorem{definition}[theorem]{Definition}

\newtheorem{example}[theorem]{Example}
\journal{}
\begin{document}
\begin{frontmatter}
\title{An introduction to the Equiangular algorithm}
\author[sad1]{Danial Sadeghi\corref{cor1}}
\cortext[cor1]{Corresponding author. E-mail: $\mathtt{dl.sadeghi@math.uk.ac.ir,~dl.sadeghi@gmail.com}$.}
\author[sad1]{Azim Rivaz}
\ead{arivaz@uk.ac.ir}
\address[sad1]{Department of Mathematics, Shahid Bahonar University of Kerman, Kerman 76169-14111, IRAN}
% % % % % % % % % % % % % % % % % % % % %
\begin{abstract}
In this paper an algorithm for constructing a set of equiangular vectors is presented. Actually, if $\{a_1,a_2,\ldots,a_k\}$ is a set of linearly independent vectors in $\mathbb{R}^n$, then the aim is to generate $k$ linearly independent vectors $\{s_1,s_2,\ldots,s_k\}$ via them, where the angles of $s_i$'s are mutually the same. Therefore a usual type of matrix decomposition is derived like the $QR$ algorithm. Then we discuss some properties of matrices derived from the new algorithm. Also the inverse and eigenvalue problems of them are studied. Then, we derive some canonical forms and applications based on the algorithm.
\end{abstract}
\end{frontmatter}
%\vspace{-0.5cm}
%%%%%%%%% Section 1 %%%%%%%%%%%
\section{Introduction} \label{intro}
We know using the Gram-Schmidt (GS) algorithm, one can construct a set of orthogonal vectors $\{q_1,q_2,\ldots,q_k\}$ via a set of linearly independent vectors successively, so that they span the same subspaces \cite{Bjorck1,Bjorck2, Schmidt, Soliverez, Terefethen}. In this paper a generalization of GS algorithm is presented. That means from $k$ linearly independent vectors $\{a_1,a_2,\ldots,a_k\}$ of a $n-$dim space $(k\leq n)$, with a prescribed cosines of the angle $\alpha=\cos\theta\in(\frac{-1}{n-1},1)$, a new set of $k$ linearly independent vectors $\{s_1,s_2,\ldots,s_k\}$ is generated through $a_i$'s, where $s_i^T s_j=\alpha, (i\neq j)$. Since for $\alpha=0$ all results overlap with GS algorithm, then we try to far from $\alpha=0$.\\

The purpose of this paper is to generalize the GS algorithm by providing an algorithm that generate linearly independent unit vectors $s_1,s_2,\ldots,s_n$ using $a_1,a_2,\ldots,a_n$, provided that $s_i^Ts_j=\alpha\in(\frac{-1}{n-1},1)$ for $i\neq j$ and $s_i^T s_i=1$. It means the $s_i$'s are equiangular. In this paper we deal with subspaces of $\mathbb{R}^n$ with the standard inner product and Euclidean norm. Throughout this paper, $\Vert X\Vert$ denotes the 2-norm of matrix $X$, $\Vert x\Vert$ denotes the Euclidean norm of vector $x$ and $\alpha=\cos\theta$. Also $e=[1,\ldots ,1]^T\in\mathbb{R}^n$ is the vector of ones and $\{e_1,e_2,\ldots ,e_n\}$ is the standard orthogonal basis of $\mathbb{R}^n$. Moreover, All matrices in this paper are real. For the sake of simplicity, we denote the transpose of the inverse of a nonsingular matrix $A$ as $A^{-T}$. Indeed, any set of equiangular vectors in $\mathbb{R}^n$ can be considered as a set of equiangular lines (ELs), but not inversely, in general. The discussion of ELs has been of interest for about sixty years of investigation. In 1973, Lemmens and Seidel \cite{Lemmens} did a comprehensive study on real equiangular line sets which is still a fundamental piece of work. In this paper, we turn our attention to equiangular vectors.
Now a general representation of an equiangular basis set in $\mathbb{R}^n$ analogous to $\{e_1,e_2,\ldots,e_n\}$ can be introduced.
\begin{proposition}
For a given $\alpha=\cos\theta\in(\frac{-1}{n-1},1)$ and $x=\frac{\sqrt{(1-\alpha)(1+(n-1)\alpha)}-1}{(1-\alpha)(n-1)-1}$, Let $v_i=xe+(1-x)e_i$ for $i=1,\ldots,n$. Then $\{s_1,\ldots,s_n\}$ are equiangular vectors with $\alpha$, where $s_i=v_i/\Vert v_i\Vert$. Moreover, if $n-1$ entries of all $s_i$ are the same as $t$, $|t|\in(0,\tfrac{1}{\sqrt{n}})$ and the last one $s=\sqrt{1-(n-1)t^2}$, then \label{P_s,t}
\begin{equation}
s=\tfrac{\sqrt{(1-\alpha)(n-1)(n-2)+n+2(n-1)\sqrt{(1-\alpha)(1+(n-1)\alpha)}}}{n} \quad ,\quad t=\tfrac{\sqrt{\alpha n+2(1-\alpha -\sqrt{(1-\alpha)(1+(n-1)\alpha)})}}{n}. \label{s,t}
\end{equation}
\end{proposition}
\begin{proof}
It can be checked that $s_{i}^{T}s_j=\dfrac{2x+(n-2)x^2}{1+(n-1)x^2}=\alpha,~(i\neq j)$. For the next part we have $s^2+(n-1)t^2=1$ and also $\alpha=2ts+(n-2)t^2$, so the Equation \eqref{s,t} is derived.
\end{proof}
\begin{proposition}
Any set of equiangular vectors $\{s_1,\ldots,s_n\}$ in $\mathbb{R}^n$ with $\alpha\in(\frac{-1}{n-1},1)$ is linearly independent.
\end{proposition}
\begin{proof}
let $x_1 s_1+\ldots+x_n s_n =0$, for arbitrary scalars $x_i$. Premultiplying this equality by $s_1^T,~s_2^T,\dots,~s_n^T$ respectively, gives a system of linear equations as $Gx=0$, where $G$ has ones on the main diagonal and $\alpha$ on all off-diagonals and $x=[x_1,x_2,\ldots,x_n]^T$. Infact, $G$ is the nonsingular Gram matrix of matrix $S=[s_1,\ldots ,s_n]$ \cite{Horn, Lanckriet}, so has no zero eigenvalue and $x\equiv0$. Note that if the upper bound of $\alpha$ holds, then $G=ee^T$ which is rank-one and singular. If the lower bound of $\alpha$ holds, then $G$ is singular with the eigenpair $(0,e)$ with algebraic multiplicity one at $0$. So the rank of both $G$, $S$ is $n-1$, so $s_1,\ldots,s_n$ are equiangular but ``linearly" dependent. Actually they are famous as equiangular lines \cite{Casazza}.
\end{proof}
This paper is organized as follows. In section \ref{equi3}, using the basis $\{e_1,e_2\ldots,e_n\}$, we construct the so-called standard equiangular basis $\{s_1,s_2\ldots,s_n\}$ with the angle $\theta$. Next we construct the equiangular vectors of $n-$dim case using a set of given linearly independent vectors. Then we can write $A=S.R$ so that the column vectors of $A$ and $S$ are the input and output vectors of the mentioned algorithm, respectively and $R$ is an upper triangular matrix and we show that the inverse of equiangular matrices can be computed with the computational complexity of $O(n^2)$. Also, we provide a proper bound for the eigenvalues of an equiangular matrix. In section \ref{schur}, we introduce some matrix factorizations based on the algorithm. Finally, In section \ref{doubequi}, we study the new special set of normal matrices named ``doubly equiangular matrices", which are equiangular like column-wise and row-wise.
\begin{figure}
\centering
\includegraphics[scale=.2]{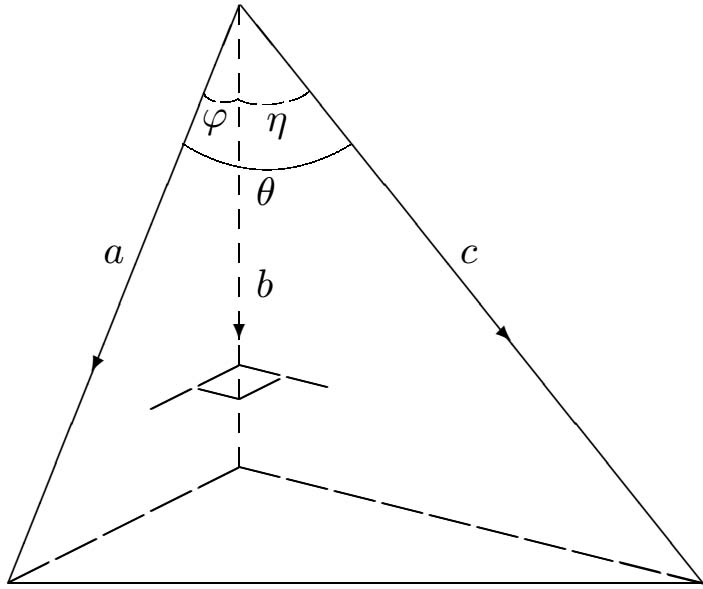}
\caption{\scriptsize{Three vectors $a,b$ and $c$ form a pyramid, so that two faces which corresponding to $\varphi$ and $\eta$ are perpendicular.}}
\label{fig1}
\end{figure}
\begin{lemma} 
Suppose that $s_1,s_2,\ldots,s_k$ are equiangular normalized vectors in $\mathbb{R}^n$ $(k\leq n)$ with $\alpha\in(\frac{-1}{n-1},1)$. Then $\Vert \sum_{i=1}^{k} s_i\Vert=\sqrt{k\left(1+(k-1)\alpha \right)}$. \label{sumsi}
\end{lemma}
\begin{proof}
The proof is clear by obtaining $\Vert \sum_{i=1}^{k} s_i\Vert^2=(\sum_{i=1}^{k} s_i)^T(\sum_{i=1}^{k} s_i)=k+k(k-1)\alpha= k(1+(k-1)\alpha)$.
\end{proof}
\begin{lemma}
The equiangular vectors $s_1,\dots , s_k$ are geven. Suppose that $\varphi$ and $\eta$, are the angles of two pair of vectors $\sum_{i=1}^{k} s_i,~s_{k+1}$ and $\sum_{i=1}^{k} s_i,~s_j$, $(1\leq j\leq k)$ respectively. Then
\begin{equation} \cos\varphi =\frac{\sqrt{k}\alpha}{\sqrt{\left(1+(k-1)\alpha \right)}}~~,~~ \cos\eta =\frac{1+(k-1)\alpha}{\sqrt{k\left(1+(k-1)\alpha \right)}}~~ \text{and}~~ \cot\varphi =\dfrac{\sqrt{k}\alpha}{\sqrt{(1-\alpha)(1+k\alpha)}}. \label{sincot}
\end{equation}
Also multiplying the above equalities together implies that
\begin{equation} \cos\varphi\cdot\cos\eta =\cos\theta =\alpha ~. \label{tri}
\end{equation} \label{Lcot}
\end{lemma}
\begin{proof}
The proof is clear from Lemma \ref{sumsi}.
\end{proof}
Indeed $\varphi$, $\eta$ and $\theta$ are the angles of a trihedral angle, in such a way that two faces corresponding to $\varphi$ and $\eta$ are perpendicular together. Actually \eqref{tri} represents the first law of cosine in the spherical pyramid \cite{Todhunter}:
\begin{lemma}
Let $a$, $b$ and $c$ are three linearly independent vectors in $\mathbb{R}^n$ $(n\geq 3)$ which make a trihedral angle with the angles $\varphi$, $\eta$ and $\theta$ and also the dihedral angle between two faces opposite to $\theta$ is $\pi/2$ as illustrated in Figure \ref{fig1}. Then $\cos \theta =\cos \varphi~\cos \eta$.
\end{lemma}
\begin{proof}
\begin{equation}
\cos \theta =\frac{a^Tc}{\Vert a\Vert~\Vert c\Vert}=\frac{\left(\text{proj}_ba+(a-\text{proj}_ba)\right)^T \left(\text{proj}_b c+(c-\text{proj}_bc)\right)}{\Vert a\Vert~\Vert c \Vert}=\frac{(\text{proj}_b a)(\text{proj}_b c)}{\Vert a\Vert ~\Vert c\Vert}= \cos\varphi~\cos\eta ~.
\end{equation}
\end{proof} 
\begin{lemma}
The equiangular vectors $s_1,\dots , s_k$ are geven. If $\mathbf{\prod}_k$ is the set of all vectors in $\mathbb{R}^{k+1}$ which make the same angle with $s_i$'s, then $\mathbf{\prod}_k$ is a subspace with dim$(\mathbf{\prod}_k)=2$. \label{Ldim2}
\end{lemma}
\begin{proof}
The set $\mathbf{\prod}_k$ can be illustrated as follows
\begin{equation}
\mathbf{\prod}_k=\bigcap~\big\{v~|~v\in\text{span}\langle s_1,\ldots, s_k,a_{k+1}\rangle, v\perp(s_i-s_j)~;~1\leq i<j\leq k\big\}. \label{P_k}
\end{equation}
So it can be seen that $\mathbf{\Sigma}_k =\langle\bigcup\{s_i -s_j\}\rangle$ is the orthogonal complement of $\mathbf{\prod}_k$. We claim that the equiangular set $\mathcal{S}_{\mathbf{\Sigma}_k}=\{s_1-s_2,s_1-s_3,\ldots,s_1-s_k\}$ can be considered as a basis of $\mathbf{\Sigma}_k$. To check the linearly independence let $\sum_{i=2}^k x_{i}(s_1-s_i)=0$ for arbitrary $x_i$. Because of linearly independence of $s_1$ through $s_k$, $(\sum_{i=2}^k x_{i})s_1 -\sum_{i=2}^k x_i s_i=0$ implies $x_i=0$ for each $i$. To prove ${\mathbf{\Sigma}_k}= \text{span}\langle\mathcal{S}_{\mathbf{\Sigma}_k}\rangle$ it is enough to show that ${\mathbf{\Sigma}_k}\subseteq\text{span}\langle\mathcal{S}_ {\mathbf{\Sigma}_k}\rangle$. Let $w\in{\mathbf{\Sigma}_k}$, then from the definition
\begin{equation}
w=\sum_{i<j}^{k}x_{i,j}(s_i-s_j)=\sum_{i<j}^{k}x_{i,j}\left[(s_1-s_j)-(s_1-s_i)\right]=\sum_{i=2}^k y_{i}(s_1-s_i),
\end{equation}
so $w\in\text{span}\langle \mathcal{S}_ {\mathbf{\Sigma}_k} \rangle$. This argument yields that dim$({\mathbf{\Sigma}_k})=k-1$, that implies dim$(\mathbf{\prod}_k)=k+1-(k-1)=2$ and $\mathbf{\prod}_k$ is a plane.
\end{proof}
%------------------------
%%%%%%%%% Section 2 %%%%%%%%%%%
\section{An algorithm for producing a set of equiangular vectors}\label{equi3}
Suppose that $\{a_1,a_2,\ldots,a_n\}$ is a set of given linearly independent vectors in $\mathbb{R}^n$. In this section we drive a set of $n$ equiangular vectors with angle $\theta\in(0,\arccos(\frac{-1}{n-1}))$ using the input vectors. First, we normalize $a_1$ as $s_1=\Vert a_1\Vert^{-1} a_{1}$~. Suppose that the vector $a_2$ makes an acute angle with $a_1$. Fom Fig. \ref{fig2} we have $u_2=a_2-s_1s_1^Ta_2$, also $b=\cot\theta~\Vert u_2\Vert~s_1$ and $v_2=a_2-(s_1s_1^Ta_2-b)$. Therefore, $v_2=u_2+\cot\theta~\Vert u_2\Vert~s_1$ . If we let $q_2 =\Vert u_2\Vert ^{-1}u_2$ which is orthogonal to $s_1$, then we have
\begin{equation}
v_2=q_2+\cot\theta ~s_1, \label{2dim}
\end{equation}
so that the next normalized equiangular vector spanned by $\{ a_1,a_2\}$ is $s_2=v_2/\Vert v_2\Vert$. For the case of obtuse angles between $a_1$ and $a_2$ the same method with some slight differences can be derived so that $\cot\theta<0$. The main theorem is illustrated as follows
\begin{figure}
\centering
\includegraphics[scale=.27]{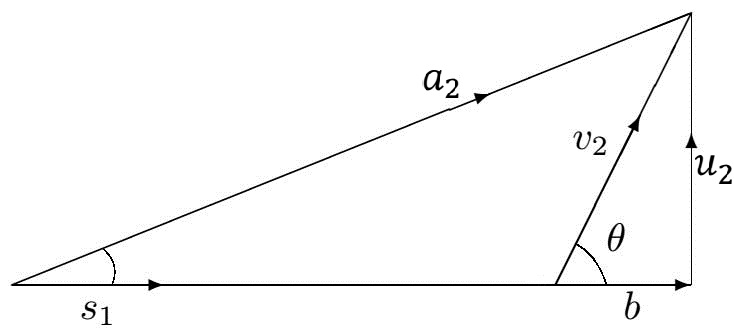}
\caption{}
\label{fig2}
\end{figure}
\begin{theorem}
Suppose that $a_1,\ldots,a_n\in\mathbb{R}^n$ are the linearly independent vectors. For a given $\alpha=\cos\theta\in(\frac{-1}{n-1},1)$, the equiangular vectors $s_1,s_2$ generated via $a_1,a_2$ by \eqref{2dim} are proposed. If $q_{k+1}\in span\langle s_1,\ldots,s_k,a_{k+1}\rangle$ is the normalized orthogonal vector to the vectors $s_1,~s_2,\dots,~s_k$, $(2\leq k< n)$. Then the $(k+1)$th equiangular vector $s_{k+1}=v_{k+1}/\Vert v_{k+1}\Vert$ is obtained as follows
\begin{equation}
v_{k+1}=q_{k+1}+\left(\tfrac{\sqrt{k}\alpha}{\sqrt{(1-\alpha)(1+k\alpha)}}\right)\frac{\sum_{i=1}^{k}s_{i}}{\Vert\sum_{i=1}^{k}s_{i}\Vert}.
\end{equation}
\end{theorem}
\begin{proof}
From Lemma \ref{Ldim2}, $\mathbf{\prod}_2=\{v~|~v\in\text{span}\langle s_1,s_2,a_3\rangle, v\perp (s_1-s_2)\}$ and $s_1-s_2$ is the orthogonal complement of $\mathbf{\prod}_2$ with dim$(\mathbf{\prod}_2)=2$. Let $\bar{a}_3$ be the projection of $a_3$ onto $\mathbf{\prod}_2$. Indeed it makes the same angle with $s_1,s_2$. We have
\begin{equation} \bar{a}_3=\text{proj}_{\mathbf{\prod} _2} {a_3}=a_3-\text{proj}_{s_1-s_2} {a_3} \label{pro}
\end{equation}
From \eqref{tri}, the third equiangular vector $v_3$ with $s_1$ and $s_2$, definitely lies in $\mathbf{\prod}_2$ ($v_3\in \mathbf{\prod}_2$). Let $\varphi$ be the angle between $v_3$ and $s_1+s_2$. Moreover, $\cot\theta$ in \eqref{2dim} is replaced by $\cot\varphi$. From Lemma \ref{Lcot}, $\cot\varphi=\dfrac{\sqrt{2}\alpha}{\sqrt{(1-\alpha)(1+2\alpha)}}$.\\

It is supposed to derive $v_3$ using $\bar{a}_3$. In comparison with the previous step \eqref{2dim}, $\bar{a}_3$ and $s_1+s_2$ can be considered as $a_2$ and $s_1$, respectively. Let $u_3=\bar{a}_3-\dfrac{(s_1+s_2)(s_1+s_2)^T\bar{a}_3}{\Vert s_1+s_2 \Vert^{2}}$. Then from \eqref{pro}
\begin{align}
u_3&=a_3-\text{proj}_{s_1-s_2}{a_3}-\frac{(s_1+s_2)(s_1+s_2)^T(a_3-\text{proj}_{s_1-s_2}{a_3})}{\Vert s_1+s_2\Vert ^{2}}\nonumber\\&=a_3-\text{proj}_{s_1-s_2}{a_3}-\dfrac{(s_1+s_2)(s_1+s_2)^Ta_3}{\Vert s_1+s_2 \Vert^{2}}=a_3-(\text{proj}_{s_1-s_2}{a_3}+\text{proj}_{s_1+s_2}{a_3}).
\end{align}
Note that $u_3$ is orthogonal to $s_1,s_2$ because $s_1-s_2\perp s_1+s_2$. Therefore from \eqref{2dim} the desired vector $v_3$ is in direction to $u_3+\cot\varphi ~\Vert u_3\Vert ~\dfrac{s_1+s_2}{\Vert s_1+s_2\Vert}$. Without loss of generality we can write
\begin{equation}
v_3=q_3 +\cot\varphi ~\dfrac{s_1+s_2}{\Vert s_1+s_2\Vert}, \label{3dim}
\end{equation}
where $q_3=\Vert u_3\Vert ^{-1}u_3$ which is orthogonal to $s_1 ,s_2$. So the next normalized equiangular vector to $s_1,s_2$ spanned by $\{ s_1,s_2,a_3\}$ is $s_3=v_3/\Vert v_3\Vert$. By induction After $k$ steps $(k\geq 3)$, for obtaining $s_{k+1}$ we define $\mathbf{\prod}_k$ as \eqref{P_k}. Like the previous steps the desired vector $\bar{a}_{k+1}$ which is the projection of $a_{k+1}$ onto the subspace $\mathbf{\prod}_k$ and then the equiangular vector $s_{k+1}$ with $s_1,~s_2,~,\dots,s_k$, will be obtained after. For simplicity, we project $a_{k+1}$ onto the orthogonal subspace of $\mathbf{\prod}_k$, say $\mathbf{\Sigma}_k =\langle\bigcup\{s_i -s_j\}\rangle$, then using \eqref{pro} in case of $(k+1)-$dim, implies
\begin{equation}
a_{k+1}= \text{proj}_{\text{span}\langle a_1,\ldots ,a_{k+1}\rangle}{a_{k+1}}= \text{proj}_{\mathbf{\Sigma}_k \oplus \mathbf{\prod}_k}{a_{k+1}}=\text{proj}_{\mathbf{\Sigma}_ k}{a_{k+1}}+\text{proj} _{\mathbf{\prod}_k}{a_{k+1}}= \text{proj}_{\mathbf{\Sigma}_ k}{a_{k+1}}+\bar{a}_{k+1},
\end{equation}
and then
\begin{equation} 
\bar{a}_{k+1} =a_{k+1}-\text{proj}_{\mathbf{\Sigma}_k}{a_{k+1}}. \label{bar}
\end{equation}
For computing $\text{proj}_{\mathbf{\Sigma}_k}{a_{k+1}}$ we derive an orthogonal basis $\mathcal{B}_{\mathbf{\Sigma}_k}=\{ \mathcal{B}_{i}\}_{i=1}^{k-1}$ of ${\mathbf{\Sigma}_k}$ from the basis $\mathcal{S}_{\mathbf{\Sigma}_k}= \{\mathcal{S}_{i}\}_{i=1}^{k-1}=\{s_1-s_2,s_1-s_3,\ldots,s_1-s_k\}$ as illustrated in Lemma \ref{Ldim2}. It can be checked that ${\mathbf{\Sigma}_k}$ is an equiangular set by the cosines of the angle $1/2$. We define an upper triangular matrix $\mathcal{R}_k=(r_{ij})$ of size $k-1$ as $r_{ii}=i$ and $r_{ij}=-1~(i<j)$. Now it can be checked that if $\mathcal{R}_k$ is multiplied by the vectors of $\mathcal{S}_{\mathbf{\Sigma}_k}$ from right, then the column vectors of obtained matrix is orthogonal: (This is mentioned in the Example \ref{ExQR} of the third section).
\begin{flushleft}
If $k=2$, then $\mathcal{B}_ {\mathbf{\Sigma}_2}=\mathcal{S} _{\mathbf{\Sigma}_2}= \{s_1-s_2\}$.\\
If $k=3$, then $\mathcal{B}_ {\mathbf{\Sigma}_3}=\{s_1-s_2, s_1+s_2-2s_3\}$.\\
If $k=4$, then $\mathcal{B}_ {\mathbf{\Sigma}_4}=\{s_1-s_2, s_1+s_2-2s_3, s_1+s_2+s_3-3s_4\}$.\\
\end{flushleft}
The general form of $\mathcal{B}_ {\mathbf{\Sigma}_k}$ for $k=2,\ldots ,n-1$ is illustrated as follows
\begin{equation} 
\mathcal{B}_{\mathbf{\Sigma}_k}=\big \{s_1-s_2,s_1+s_2-2s_3~,\ldots,~\sum_{j=1}^{k-1}s_{j} -(k-1)s_k\big \}.
\end{equation}
From \eqref{bar} $\bar{a}_{k+1}=a_{k+1}-\sum_{i=1}^{k-1}\text{proj}_ {\mathcal{B}_i}{a_{k+1}}$. Moreover from Lemma \ref{Lcot}, $\cot\varphi =\dfrac{\sqrt{k}\alpha}{\sqrt{(1-\alpha)(1+k\alpha)}}$.
\\Let $u_{k+1}=a_{k+1}-\left(\text{proj}_{\sum_{i=1}^ {k}s_{i}}{a_{k+1}}+\sum_{i=1}^{k-1} \text{proj}_{\mathcal{B}_i}{a_{k+1}} \right)$. Note that $\{\mathcal{B}_1,\ldots ,\mathcal{B}_{k-1},\sum _{i=1}^k s_i\}$ is also a set of orthogonal vectors with the same span as $\{s_1,\ldots ,s_k\}$. Then $u_{k+1}$ is orthogonal to $s_i$'s. Now we substitute $\sum_{i=1}^{k}s_i$ and $\sum_{i=1}^{k-1}\text{proj}_{\mathcal{B} _i}{a_{k+1}}$ instead of $s_1+s_2$ and $\text{proj}_{s_1-s_2}{a_3}$ in \eqref{3dim}, respectively to obtain $v_{k+1}$ in this case. So we have $v_{k+1}=u_{k+1}+\cot\varphi ~\Vert u_{k+1}\Vert ~ \dfrac{\sum_{i=1}^{k}s_{i}}{\left\Vert\sum_{i=1}^{k}s_{i}\right\Vert}$ and without loss of generality 
\begin{equation}
v_{k+1}=q_{k+1}+\cot\varphi ~\frac{\sum_{i=1}^{k}s_{i}}{\left\Vert\sum_{i=1}^{k}s_{i}\right\Vert}, \label{k+1}
\end{equation}
where $q_{k+1}=\Vert u_{k+1}\Vert ^{-1}u_{k+1}$ is orthogonal to $s_1,\ldots ,s_k$. The $(k+1)$th normalized equiangular vector spanned by $\{a_1,\ldots ,a_k\}$ is $s_{k+1}=v_{k+1}/\Vert v_{k+1}\Vert$. We can proceed this method to obtain $n$ equiangular vectors in $\mathbb{R}^n$.
\end{proof}
\begin{algorithm}
\caption{.~(Equiangular Algorithm). This algorithm produces equiangular vectors $s_i\in\mathbb{R}^n$ with the cosine of the angle $\alpha\in(\frac{-1}{n-1},1)$ via the linearly independent vectors $a_k$, for $k=1:n$}
\label{algo:equ}
\begin{algorithmic}[1]
\STATE $s_1=\tfrac{a_1}{\Vert a_1\Vert}~;~\bar{s}=\text{zeros}(n,1)~;~r=0~;$
\FOR {$~k=2:n$}
\STATE $\bar{s}=\bar{s}+ s_{k-1}~;~ s=\tfrac{\bar{s}}{\Vert\bar{s}\Vert}~;$
\STATE
$q=\Vert v_{k}-\left(s s^{T}a_k+r\right)\Vert ^{-1}\cdot\left(a_{k}-(ss^{T}a_k+r)\right)~;$
\STATE $v_{k}=q+ \tfrac{\sqrt{k}\alpha}{\sqrt{(1-\alpha)(1+k\alpha)}}s~;$
\STATE $s_{k}= \tfrac{v_k}{\Vert v_k\Vert}~;~r=0~;$
\STATE $\mathcal{B}_{k-1}= \left(\sum_{j=1}^{k-1} s_{j}\right)-(k-1) s_k~;$
\FOR {$~i=1:k-1$}
\STATE $r=r+\tfrac{\mathcal{B}_i{\mathcal{B}_i}^{T}a_k}{{\Vert \mathcal{B}_i\Vert}^2}~;~$
\ENDFOR
\ENDFOR 
\end{algorithmic}
\end{algorithm}
We present an algorithm named Equiangular Algorithm (EA) for this method. The outcome matrix of this method $S=[s_1,\ldots,s_m]$ of size $n\times m$, is called $\mathit{Equiangular~matrix}$. We denote the set of all full-rank equiangular matrices of size $n\times m$ with $\alpha=\cos\theta$, $\alpha\in(\frac{-1}{n-1},1)$ by $EM^{n\times m}_\alpha$ and also the corresponding nonsingular square ones by $EM^n_\alpha$.
\\\textbf{Orthogonal vectors.} The formula \eqref{k+1} reduces to the orthogonalization process if $\alpha =0$. In this case at the step k+1, we have
\begin{equation}
a_{k+1}=a_{k+1}-(\text{proj}_ {\mathcal{B}_1}{a_{k+1}}+ \ldots +\text{proj}_{\mathcal {B}_{k-1}}{a_{k+1}}+ \text{proj}_{\mathcal{\bar B}_k}{a_{k+1}}),\quad k=1,2,\ldots ,n-1, \label{alpha=0}
\end{equation}
where $\mathcal{\bar B}_k=\sum_ {i=1}^{k}s_i$. It can be said that the normalized orthogonal vector $s_{k+1}$ in case of $\alpha=0$, is constructed using the set of orthogonal vectors $\{\mathcal{B}_1,\cdots,\mathcal{B}_{k-1},\mathcal{\bar B}_k\}$ instead of the previous set of orthonormal vectors $\{s_1,\ldots ,s_k\}$ which is unlike the Gram-Schmidt algorithm.
\\Below is an example for outcome of this algorithm using MATLAB.
\begin{example}
For the Vandermonde matrix $A=\left[\footnotesize{\begin{array}{cccc}1&1&1&1\\1&2&4&8\\1&3&9&27\\1& 4&16&64\\\end{array}}\right]$, applying EA with $\theta _1=\pi/3$ and $\theta _2=\pi/4$ gives $S_1=\left[\footnotesize{\begin{array}{cccc} 0.5&-0.3309&~~0.4646&-0.3352\\ 0.5&~~0.0564&-0.2228&~~0.4469\\0.5& ~~0.4436&-0.0937&-0.7428\\0.5&~~
0.8309&~~0.8519&-0.3689\end{array}}\right]$ and $S_2=\left[\footnotesize{\begin{array}{cccc}0.5&-0.1208&~~0.4789&-0.3889\\0.5& ~~0.1954&-0.0337&~~0.2190\\0.5&~~0.5117& ~~0.0972&-0.7376\\0.5&~~0.8279&~~0.8718& -0.5067\end{array}}\right]$, respectively. \quad$\lozenge$
\end{example}
The process of producing the equiangular vectors provides a matrix factorization by generating an upper triangular matrix within EA algorithm, like QR decomposition.
\begin{theorem}
Suppose $A\in\mathbb{R}^{n\times m}$  has full-rank $(m\leq n)$. For given $\alpha\in(\frac{-1}{n-1},1)$ there is an equiangular matrix $S\in EM^{n\times m}_\alpha$ and an upper triangular $R\in\mathbb{R}^{m\times m}$ so that $A=SR$. \label{SR}
\end{theorem}
\begin{proof}
In step $k$ ($k<m$) of EA which is applied to the column vectors of $A=[a_1,\ldots,a_m]$, we derive a representation of $v_k$ followed by \eqref{k+1} as $\Vert v_k \Vert\cdot s_k=a_k+\sum_{i=1}^{k-1} d_{i} s_{i}$. So
\begin{equation}
a_{k}=-\sum_{i=1}^{k-1}d_{i}s_{i}+\Vert v _k\Vert s_{k}=[s_1,\ldots ,s_k]{\small\begin{bmatrix}-d_1\\ \vdots\\-d_{k-1}\\\Vert v_{k}\Vert\end{bmatrix}}=S_k R_{k}~.
\end{equation}
By EA $S_k$ can be extended as $S=\begin{bmatrix}S_k,\bar{S}_{m-k}\end{bmatrix}$, where $\bar{S}_{m-k}$ is a $n\times (m-k)$ equiangular matrix and then
\begin{equation}
a_k=\begin{bmatrix}S_k,\bar{S}_{m-k}\end{bmatrix} \begin{bmatrix} R_k\\0\end{bmatrix}=S\begin{bmatrix} R_k \\0\end{bmatrix}.
\end{equation}
This is the $k$th column vector in both sides of $A=SR$.
\end{proof}
\begin{example}
For the Minij matrix $A=\left[\footnotesize{\begin{array}{cccc}1&1&1&1\\1&2&2&2\\1&2&3&3\\1&2 &3&4\end{array}}\right]$, Theorem \ref{SR} results $A=S_1 R_1$ by $\theta _1=\pi/6$ where\\ $S_1=\left[\footnotesize{\begin{array}{cccc}0.5&0.0000& 0.2321&0.2321\\0.5&0.5774&0.1384&0.3854\\ 0.5&0.5774&0.6808&0.2603\\0.5&0.5774& 0.6808&0.8543\end{array}}\right]$ and $R_1=\left[\footnotesize{\begin{array}{cccc}2.0000&2.0000&1.1444&0.1830\\0~\qquad& 1.7321&2.0312&1.6471\\0~\qquad&0~\qquad& 1.8436&2.2317\\0~\qquad&0~\qquad&0~\qquad& 1.6834\end{array}}\right]$ and also results $A=S_2 R_2$ by $\theta _2=3\pi/8$ where $S_2=\left[\footnotesize {\begin{array}{cccc} 0.5&-0.6088~~&-0.0301~~&-0.0301~~\\ 0.5&0.4580&-0.4597~~&0.1080\\0.5&0.4580& 0.6276&-0.2691~~\\0.5&0.4580&0.6276&0.9566 \end{array}}\right]$ and $R_2=\left[\footnotesize{\begin{array}{cccc}2.0000&3.1413&3.6476&3.7239\\0~\qquad& 0.9374&1.3078&1.3161\\0~\qquad&0~\qquad&0.9197& 1.2027\\0~\qquad&0~\qquad&0~\qquad&0.8159 \end{array}}\right].$~ $\lozenge$
\end{example}
\begin{corollary}
Any symmetric positive definite (spd) matrix $A$ can be factorized as a Cholesky factorization plus a rank-one matrix.
\end{corollary}
\begin{proof}
Suppose $A\in\mathbb{R}^{n\times n}$ is a real spd matrix, then $A$ can be written as $B^{T}B$, where $B$ is a nonsingular matrix. SR factorization by Theorem \ref{SR} implies 
\begin{align}
A=&B^{T}B=R^{T}S^{T}SR=R^{T}{\footnotesize \left[\begin{array}{ccc}1&\cdots &\alpha\\\vdots &\ddots &\vdots\\\alpha &\cdots &1\end{array} \right]}R=R^{T}\left[(1-\alpha)I_n+\alpha ee^T\right] R\nonumber\\=&(1-\alpha)R^{T}R+ \alpha(R^{T}e)(R^{T}e)^T=Cholesky+rank\text{-}one.
\end{align}
\end{proof}
It is notable that this factorization is not unique with respect to a fixed angle. Actually, the number of SR factorization of a matrix $A\in\mathbb{R}^{n\times m}$ of full rank, is $2^{n-1}$. Because by altering the sign of vector $q_2$ in $\eqref{2dim}$ there are two choices for constructing $s_2$ in step 2. Also by altering the sign of vector $q_3$ in \eqref{3dim} there are two choices for constructing $s_3$ in step 3. Therefore we have two choices in each step until the $n$th vector $s_n$ is derived.
%\section{Inverse and eigenvalue problems} \label{inveig}
Now we discuss the inverse and eigenvalue problems of equiangular matrices. We need to introduce the so-called $\mathit{Gram~matrix}$ \cite{Godsil,Horn}. 
\begin{definition}
If $A=[a_1,\ldots,a_m]$, then the matrix $G=A^T A$, is called the $\mathit{Gram~matrix}$ of $A$.\qquad\qquad\qquad\quad $\lozenge$ \label{defG}
\end{definition}
Suppose that $S=[s_1,\ldots,s_n]\in \text{EM}^n_\alpha$, and then we define $G_\alpha$ as
\begin{equation}
S^{T}S={\footnotesize \left[\begin{array}{cccc}1&\alpha &\cdots &\alpha\\\alpha &1&\cdots &\alpha\\\vdots &\vdots &\ddots &\vdots\\\alpha &\alpha &\cdots &1\end{array}\right]} \label{I},
\end{equation}
which is the Gram matrix of $S$. Since for $x\neq 0$ we have $x^TG_\alpha x=\Vert Sx\Vert^2> 0$, then $G_\alpha$ is positive definite. $G_\alpha$ can be rewritten as $I+\alpha\mathcal{S}$, where $\mathcal{S}$ has zeros on its main diagonal and ones on all off-diagonals and is a special case of the $\mathit{Seidel~matrix}$. 
%We have $S^{-1}S^{-T}=G_\alpha ^{-1}$ so $S^{-1}=G_\alpha^{-1}S^T$. It is yield that $G_\alpha^{-1}=\beta G_{\alpha^{\prime}}$, where
Since the fast matrix inversion is important in Linear Algebra, then we try to derive $S^{-1}$, rapidly.
\begin{proposition}
If $S\in EM_\alpha^n$, then $S^{-1}=\beta G_{\alpha^\prime}S^T$ and can be computed with $O(n^2)$ arithmetic operations where \label{invS}
\begin{equation}
\beta=\frac{1+(n-2)\alpha}{(1-\alpha)\big(1+(n-1)\alpha\big)}\qquad,\qquad {\alpha^{\prime}}=\frac{-\alpha}{1+(n-2)\alpha}~. \label{alphbet}
\end{equation}  
\end{proposition}
\begin{proof}
It is clear to see that $S^{-1}S^{-T}=G_\alpha^{-1}=\beta G_{\alpha^\prime}$, where $G_{\alpha^\prime}$ is the another Gram matrix with the mentioned $\alpha^\prime$. On the other hand let $[S]_{ij}=s_{ij}$, then it can be checked that $[S^{-1}]_{ij}=\beta(s_{ji}+\alpha^{\prime} \cdot\sum_{k\not =i}s_{jk})$. Therefore, $S^{-1}$ can be computed with $O(n^2)$ arithmetic operations.
\end{proof}
If $\alpha$ is zero, then $\beta=1$ and $G_{\alpha'}=G_0=I$, so $S^{-1}=S^T$, which shows that the result is true in case of orthogonality.
\begin{corollary}
Let $S=[s_1,\ldots,s_n]\in \text{EM}^n_\alpha$ by $\alpha\in(0,1)$, then the rows of $S^{-1}=[s_1'^T,\ldots,s_n'^T]^T$ are equiangular with $\cos\theta'=\alpha'\in(\frac{-1}{n-1},0)$, where $\Vert s_{i}'\Vert=\sqrt{\beta}$ and vice versa. Moreover, the cosine of the angle between $s_i$ and $s_{i}'$ is $1/\sqrt{\beta}$, where $\alpha',\beta$ are defined in \eqref{alphbet}.          \label{Tinv}
\end{corollary}
\begin{proof}
If $\tau_i$ is the angle between the vectors $s_i, s_{i}^{\prime T}$, then $1=s_{i}'\cdot s_i=\Vert s_{i}'\Vert\cos\tau_i$, so $\cos\tau_{i}=1/\Vert s_{i}'\Vert$ and $\tau_i\in(0,\pi/2)$. Now two vectors $s_i$ and $s_{i}^{\prime T}$ make the angles $\theta$ and $\pi/2$ with all vectors of the set $\tilde{S}_{i}=\{s_j\}_{j\neq i}$, respectively. Therefore, we can introduce the subspace $\mathbf{\prod}_{n,i}$ with respect to $i$ which is similar to $\mathbf{\prod}_k$ in \eqref{P_k} in which all the vectors have the same angle with the vectors of $\tilde{S}_i$, mutually. So $s_i,s_{i}^{\prime T}\in \mathbf{\prod}_{n,i}$. On the other hand, $z_i=\sum_{j\neq i}s_{j}\in\mathbf{\prod} _{n,i}$. Since $\mathbf{\prod}_{n,i}$ is a plane, then $z_i \in\text{span}\langle s_i,s_{i}^{\prime T}\rangle$. Clearly, $s_{i}^{~\prime T}\perp z_i$. If $\varphi_i$ is the angle between $s_i,z_i$, then from Pythagoras Theorem, $1=\cos^{2}\varphi_i+\cos^{2}\tau_i =(n-1)\alpha^2/\big(1+ (n-2)\alpha\big)+1/\Vert s_{i}^{\prime}\Vert^2$, thus
\begin{equation} 
\Vert s_{i}^{~\prime}\Vert =\sqrt{\frac{1+(n-2)\alpha}{(1-\alpha)\left(1+(n-1)\alpha\right)}}= \sqrt\beta~,\quad i=1,\ldots ,n\label{ninv}
\end{equation}
Taking inner product of rows in two sides of $S^{-1}=\beta G_{\alpha'}S^T$ gives
\begin{align} 
s_{i}'s_{j}^{\prime T} 
&=\beta^{2}[\alpha'\cdots \substack{i\text{th}\\1\\~}\cdots \alpha']~
G_\alpha~ 
[\alpha'\cdots \substack{j\text{th}\\1\\~}\cdots \alpha']^T\nonumber\\& =\beta^{2}\big[\left(\alpha +\alpha'+(n-2)\alpha \alpha'\right)[1\cdots \substack{i\text{th}\\0\\~}\cdots 1]+\left(1+(n-1)\alpha\alpha'\right)[0\cdots\substack{i\text{th} \\1\\~}\cdots0]\big][\alpha'\cdots\substack{j\text{th} \\1\\~}\cdots\alpha ']^T\nonumber\\&=\beta^{2} \big[\left(\alpha+\alpha'+(n-2)\alpha\alpha'\right)((n-2)\alpha'+1)+ \left(1+(n-1)\alpha\alpha' \right)\alpha'\big]=\frac{-\alpha}{(1-\alpha)\left(1+(n-1)\alpha\right)}~.\label{sipsjp}
\end{align}
Then $\cos\theta^\prime=\dfrac{s_{i}^{\prime} s_{j}^{\prime T}}{\Vert s_{i}^{\prime}\Vert\Vert s_{j}^{\prime T}\Vert}= \dfrac{-\alpha}{1+(n-2)\alpha}=\alpha^{\prime}$. Since $\alpha\alpha^{\prime}<0$, then the result is obtained.
\end{proof}
With regard to \eqref{ninv}, matrix $\beta^{-1/2}S^{-1}$ is row-wise equiangular, so $\beta^{-1/2}S^{-T}\in \text{EM}^n_{\alpha^{\prime}}$. for $n=2$, $\alpha^{\prime}=-\alpha$ and $\theta^{'}=\pi -\theta$ (Figure \ref{fig3}).
\begin{figure}
\centering
\includegraphics[scale=.22]{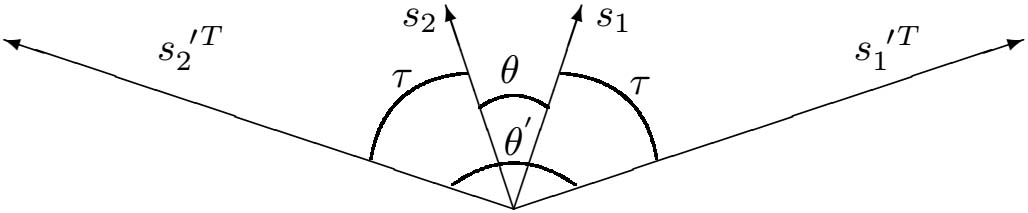}
\caption{\scriptsize{Inverse of the matrix $[s_1,s_2]$ is $[s_{1}^{~\prime},s_{2}^{~\prime}]^{T}, \text{where}~\theta +\tau =\frac{\pi}{2}$ and $\theta+\theta^{'} =\pi$.}}
\label{fig3}
\end{figure}
\begin{corollary}
If $\alpha^{\prime},\beta$ are defined as \eqref{alphbet}, then $G_\alpha G_{\alpha^{\prime}}=(1/\beta)I_n$~. \label{invG}
\end{corollary}
\begin{proof}
It is obvious from Proposition \ref{invS}.
\end{proof}
We give some examples of computing the inverse of equiangular matrices using \ref{invS}.
\begin{example}
SR-decomposition for the $4\times 4$ Hilbert matrix \textit{H} with $\theta=\pi/3$ is as follows
\\ $S=\left[\footnotesize{\begin{array}{cccc}0.8381&-0.0336&~~0.3939 &~~0.2788\\0.4191& ~~0.5921&-0.2572&~~0.4381\\0.2794&~~0.5977& ~~0.4062&-0.3031\\0.2095&~~0.5396&~~0.7834& ~~0.7991\end{array}}\right]$ and $R=\left[\footnotesize{\begin{array}{cccc}1.1932&0.6021&0.3998&0.2980\\0~\qquad& 0.1369&0.1426&0.1318\\0~ \qquad&0~\qquad&0.0076& 0.0117\\0~\qquad&0~\qquad&0~\qquad&0.0002 \end{array}}\right]$. We have $\alpha=1/2$, ${\alpha^{\prime}}=-1/4$ and $\beta=8/5$, then $S^{-1}=1.6~G_{-0.25}\left[\footnotesize {\begin{array}{cccc}~~0.8381&~~0.4191&~~0.2794 &0.2095\\-0.0336&~~0.5921&~~0.5977&0.5396\\ ~~0.3939&-0.2572&~~0.4062&0.7834 \\~~0.2788&~~0.4381&-0.3031&0.7991 \end{array}}\right]$.\qquad\qquad\qquad\qquad\qquad\qquad $\lozenge$
\end{example}
\begin{example}
SR-decomposition for the Identity matrix $I_4$ with $\theta=\pi/4$ is as follows\\ $S=\left[\footnotesize{\begin{array}{cccc}1&0.7071&0.7071&0.7071\\0&0.7071 &0.2929&0.2929\\0&0~\qquad&0.6436& 0.1885\\0&0~\qquad&0~\qquad&0.6154 \end{array}}\right]$ and $R=S^{-1}=\left[\footnotesize{\begin{array}{cccc} 1&-1\quad\qquad&-0.6436~~&-0.4760\\ 0&1.4142&-0.6436~~&-0.4760\\0&0~\qquad &1.5538&-0.4760\\0&0~\qquad&0~\qquad &~~1.6250 \end{array}}\right]$. By rounding we have $S\in \text{EM}^4_{0.7071}$ and $(0.6154)R^T\in \text{EM}^4_{-0.2929}$. Note that there is only one upper triangular equiangular matrix with positive entries with respect to $\alpha$. \qquad\qquad\qquad\qquad\qquad \qquad\qquad\qquad\qquad \qquad\qquad\qquad\qquad\qquad\qquad\quad $\lozenge$
\end{example}
\begin{example}
SR-decomposition for the orthogonal matrix $Q=\left[\footnotesize{\begin{array}{ccc}~~3/7&-2/7&~~6/7\\~~6/7&~~3/7 &-2/7\\-2/7&~~6/7&~~3/7\end{array}} \right]$ with $\theta=\pi/3$ is as follows
\\$S=\left[\footnotesize{\begin{array}{ccc}~~0.4286&-0.0332&0.8317\\~~0.8571 &~~0.7997&0.3190\\-0.2857&~~0.5995& 0.4545\end{array}}\right]$ and $R=\left[\footnotesize{\begin{array}{ccc}1& -0.5774&-0.4082\\0&~~1.1547&-0.4082\\ 0&0~~~~~&~~1.2247\end{array}}\right].$
  In regard to \eqref{alphbet}, $\beta^{-1/2}R$ is row-equiangular: $RR^T =S^{-1}S^{-T}=G_ \alpha ^{-1}=\beta G_{\alpha^{\prime}}$, with $\alpha=1/2$. \label{ExQR} \qquad\qquad\qquad\qquad\qquad\qquad \qquad\qquad\qquad\qquad\quad $\lozenge$
\end{example}
Now we discuss the eigenvalues of equiangular matrices. Actually, we present lower and upper bounds for the eigenvalues of an Equiangular matrix $S$ which is relative to the eigenvalues of its corresponding matrix $G_\alpha$~.
\begin{lemma}
Suppose that $G_\alpha$ is the Gram matrix of a given matrix $S\in \text{EM}^n_\alpha$. The eigenvalues of $G_\alpha$ are $1-\alpha$ and $1+(n-1)\alpha$. \label{Leig}
\end{lemma}
\begin{proof}
Since $G_\alpha e=\alpha e e^{T}e+(1-\alpha)I_n e=(1+(n-1)\alpha)e$, then $(1+(n-1)\alpha,e^{T})$ is an eigenpair of $G_\alpha$. Let $x=[x_1,\ldots,x_n]^T$, with $x_1+\ldots+x_n=0$, so
$G_\alpha x=\alpha ee^{T} x+(1-\alpha) I_n x=(1-\alpha)x$. Thus, $1-\alpha$ is the second eigenvalue of $G_\alpha$, with the algebraic multiplicity $n-1$. Therefore, $\sigma(G_\alpha)=\{1-\alpha,1+(n-1)\alpha\}$. 
\end{proof}
As noted in \eqref{P_s,t}, a set of equiangular vectors $s_1,\ldots,s_n$ can be embedded into the positive coordinate axes so that $n-1$ entries of all of them are the same as $|t|\in(0,\tfrac{1}{\sqrt{n}})$ and last one is $s=\sqrt{1-(n-1)t^2}$. Since $\alpha=2ts+(n-2)t^2$, then it can be shown that
\begin{equation}
s=\tfrac{\sqrt{(1-\alpha)(n-1)(n-2)+n+2(n-1)\sqrt{(1-\alpha)(1+(n-1)\alpha)}}}{n} \quad ,\quad t=\tfrac{\sqrt{\alpha n+2(1-\alpha -\sqrt{(1-\alpha)(1+(n-1)\alpha)})}}{n}, \label{s,t1}
\end{equation}
For this reason, $s\in(\frac{1}{\sqrt{n}},1)$ and the matrix $S$ with these vectors is positive definite with positive eigenvalues. $S$ can be rewritten as $sG_{t/s}$ and from Lemma \ref{Leig} $\sigma(sG_{t/s})=\{ s-t,s+(n-1)t\}$. We let $\bar{S}_\alpha=sG_{t/s}$. It is notable that $\bar{S}_\alpha$ is the unique principal square root of $G_\alpha$; i.e., $\bar{S}_\alpha={G_ \alpha}^{1/2}$ \cite{Higham}. Since $\bar{S}_\alpha,S\in \text{EM}^n_\alpha$, then there exists an orthogonal matrix $Q$ so that $S=Q\bar{S}_\alpha$. Since $\bar{S}_\alpha$ is positive definite, $S$ is nonsingular and $\bar{S}_\alpha =(S^T S)^{1/2}$, then the last equality is the ``polar decomposition" of $S$ \cite{Higham1}. This equality can be interpreted as a transformation of the orthogonal matrices to the equiangular ones and vise versa ($Q=S\bar{S}_\alpha^{-1}$).\\
If $(\mu,x)$ is an eigenpair of $\bar{S}_\alpha$, then $(\mu^2,x)$ is an eigenpair of $G_\alpha$. So $\sigma(\bar{S}_\alpha)=\{\sqrt{1-\alpha},\sqrt{1+(n-1)\alpha}\}=\{s-t,s+(n-1)t\}$ and $s,t$ are as follows
\begin{equation}
s=\frac{\sqrt{1+(n-1)\alpha}+(n-1)\sqrt{1- \alpha}}{n}\quad ,\quad t=\frac{\sqrt{1+(n-1)\alpha}- \sqrt{1-\alpha}}{n}\cdot \label{s,t2}
\end{equation}
\begin{example}
For the Gram matrix $G_{1/2}={\footnotesize\left[\begin{array}{ccc}1&1/2&1/2\\1/2&1&1/2\\1/2 &1/2&1\end{array}\right]}$, the rounded form of the square roots of $G_{1/2}$ is dependent on scalars $s,t$ as follows
\begin{align}
\sqrt{G_{1/2}}&={\footnotesize \left[\begin{array}{ccc}0.9428&0.2357&0.2357\\0.2357& 0.9428&0.2357\\0.2357&0.2357 &0.9428\end{array}\right]},~ \text{for}~~0<t<\tfrac{1}{\sqrt{n}}~,\nonumber
\\\sqrt{G_{1/2}}&={\footnotesize \left[\begin{array}{ccc}0&0.7071&0.7071\\0.7071 &0&0.7071\\0.7071&0.7071&0 \end{array}\right]},~\text{for}~~ \tfrac{1}{\sqrt{n}}<t\leq\tfrac{1}{\sqrt{n-1}}.
\end{align}
Since $\sigma(\bar{S}_{1/2})=\{s-t,s+(n-1)t\}$, then the eigenvalues of first matrix are $0.7071,1.4142$ and those of the second one are $-0.7071,1.4142$. Since $\bar{S}_{1/2}$ must be positive definite and $s>0$, then the first one is $\bar{S}_{1/2}$. \qquad $\lozenge$\label{psquare}
\end{example}
In the next theorem, a lower and upper bounds for eigenvalues of equiangular matrices are presented.
\begin{theorem}
Let $(\lambda,x)$ be an eigenpair of an equiangular matrix $S\in \text{EM}^n_\alpha$ with $\Vert x\Vert =1$ then the following bounds hold:
\begin{equation}
\lambda _{\text{min}}(\bar{S}_\alpha)\leq\vert \lambda \vert\leq \lambda_{\text{max}} (\bar{S}_\alpha),
\end{equation}
where $\lambda_{min}$ and $\lambda_{max}$ stand for the minimum and maximum eigenvalues of $\bar{S}_\alpha$, respectively.
\end{theorem}
\begin{proof}
Since $Sx=\lambda x$, and $x^{*}S^{T}=\lambda^{*}x^{*}$, then $x^{*}S^{T}Sx=\lambda \lambda^{*} x^{*}x=\vert\lambda\vert^{2}\Vert x\Vert^{2}=\vert\lambda\vert^{2}$. On the other hand
\begin{align}
x^{*} S^{T}Sx &=x^{*}G_\alpha x=x^ {*}(\alpha ee^T+(1-\alpha)I)x=\alpha (\sum_{i=1} ^{n} x_i) (\sum_{i=1}^{n}x^{*}_i)+(1-\alpha)\Vert x\Vert^{2}\nonumber\\&=\alpha\big|\sum_{i=1}^{n}x_i\big|^{2}+1-\alpha=\alpha\vert e^Tx \vert^{2}+1-\alpha ~,
\end{align}
where  $e$ is the vector of ones. Taking these equalities together implies 
\begin{equation}
\vert\lambda\vert =\sqrt{\alpha ~\vert e^T x \vert^{2}+1-\alpha}~.\label{eig}
\end{equation}
Equation \eqref{eig} describes the relationship between eigenvalues and eigenvectors of $S$. Maximum of $\vert e^T x\vert$ is attained if $x=\pm\frac{e}{\Vert e\Vert}$ which is an eigenvector of $\bar{S}_\alpha$. So $\vert\lambda\vert\leq\sqrt{\alpha n+1-\alpha}=\sqrt{1+(n-1)\alpha} =\lambda_{\text{max}}\bar{S}_\alpha$ if $0<\alpha<1$.
Likewise, minimum of $\vert e^T x\vert$ is attained if $e^{T}x=0$ which in this case $x\in\ker(e^{T}x)$. Also, $\vert\lambda\vert\geq \sqrt{1-\alpha} =\lambda_{\text{min}}\bar{S}_\alpha$ if $0<\alpha<1$. The same result holds for the case of $\frac{-1}{n-1}<\alpha<0$. 
\end{proof} 
\begin{proposition}
The condition number of any matrix $S\in \text{EM}^n_\alpha$ relative to $2-$norm is equal to $\sqrt{1+\dfrac{n\alpha}{1-\alpha}}$ if $0<\alpha<1$ and $\sqrt{1+\dfrac{n\vert\alpha\vert}{1-(n-1)\vert\alpha\vert}}$ if $\frac{-1}{n-1}<\alpha<0$.
\end{proposition}
\begin{proof}
Since $\Vert S\Vert_2 =\sqrt{\lambda_{\text{max}}(S^T S)}$ and $\Vert S^{-1}\Vert_2 =\sqrt{\lambda_{\text{max}}(SS^T)^{-1}}=\dfrac{1}{\sqrt{\lambda_{\text{min}}(S^T S)}}$, then $\kappa_2 (S)=\sqrt{\dfrac{\lambda_{\text{max}}(S^T S)}{\lambda_{\text{min}}(S^T S)}}$. So for any case of $\alpha$ the result is obvious.
\end{proof}
Note that $S$ converges to ill-conditioning as $\alpha\rightarrow 1$ or $\alpha\rightarrow\frac{-1}{n-1}$.
%%%%%%%%%% section 4 %%%%%%%%%%
\section{Some matrix factorization} \label{schur}
\begin{proposition}
For $A\in\mathbb{R}^{n\times n}$, there exists $S\in \text{EM}^n_\alpha$ so that $S^{-1}AS=T$ is a block upper triangular, with $1\times 1$ and $2\times 2$ blocks on its diagonal. The eigenvalues of $A$ are the eigenvalues of diagonal blocks of $T$. The $1\times 1$ blocks correspond to real eigenvalues, and the $2\times 2$ blocks to pairs of complex conjugate eigenvalues.
\label{STiS}
\end{proposition}
\begin{proof}
From the Real Schur form $A=Q\Lambda Q^T$, where $Q$ is an orthogonal matrix and $\Lambda$ is a block upper triangular. From SR decomposition, we have $Q=SR$, where $S\in \text{EM}^n_\alpha$ and $R$ is upper triangular. Then $A=Q\Lambda Q^T=SR\Lambda R^{-1} S^{-1}=STS^{-1}$. Therefore, $T=R\Lambda R^{-1}$ is a block upper triangular, whose blocks are conformable with those of $\Lambda$.
\end{proof}
We want to know which matrices have equiangular eigenvectors. If for some special matrix $A$, there exists a matrix $S\in \text{EM}^n_\alpha$ so that $S^{-1}AS=T=\text{diag}(t_1,\ldots,t_n)$, then $A$ has $n$ equiangular eigenvectors. We can provide a test to check the existence of the equiangular eigenvectors for a square matrix in the next theorem. First, we intoduce the upper triangular equiangular matrices.
\begin{lemma}
There is a unique triangular equiangular $n\times n$ matrix in terms of $\alpha\in(\frac{-1}{n-1},1)$ as follows
\begin{equation}
\hat{S}=\left[\footnotesize{\begin{array}{ccccc}1&\alpha &\alpha &\cdots &\alpha\\0&\sqrt{1-\alpha^2} &\frac{\alpha(1-\alpha)}{\sqrt{1-\alpha^2}}&\cdots & \frac{\alpha(1-\alpha)}{\sqrt{1-\alpha^2}}\\0&0&\ddots\qquad &~&\vdots\\\vdots &\vdots &\ddots &~&~ \\0&0&\cdots &~&~\end{array}}\right] \label{shat}
\end{equation}
which is obtained from the SR decomposition of the Identity matrix and its entries satisfy the following
\begin{itemize}
%\begin{LTR}
\item $\hat{s}_{11}=1,$
\item $\hat{s}_{i(i+1)}=\cdots =\hat{s}_{in}=\hat{s}_{ii}-\dfrac{1-\alpha}{\hat{s}_{ii}},$
\item $\hat{s}_{ii}^2=1-(\hat{s}_{1i}^2+ \hat{s}_{2i}^2+\cdots+\hat{s}_{(i-1)i}^2).$
%\end{LTR}
\end{itemize}
\end{lemma}
\begin{proof}
Clearly $\hat{s}_{11}=1$. Since $\hat{s}_1^T\hat{s}_i=\alpha$, then $\hat{s}_{12}=\cdots=\hat{s}_{1n} =\dfrac{\alpha}{\hat{s}_{11}}$. Also, $\hat{s}_i^T\hat{s}_j=\alpha$ $(i<j)$, therefore, by induction
\begin{equation}
\hat{s}_{i(i+1)}=\hat{s}_{i(i+2)}=\cdots =\hat{s}_{in}=\frac{\alpha -(\hat{s}_{12}^2+\hat{s}_{23}+\cdots +\hat{s}_{(i-1)i}^2)}{\hat{s}_{ii}}\cdot \label{sij}
\end{equation}
On the other hand, $\Vert \hat{s}_i\Vert=1$, and then $(\hat{s}_{1i}^2+\cdots +\hat{s}_{(i-1)i}^2)=1-\hat{s}_{ii}^2$. Also, from \eqref{shat}, $\hat{s}_{i(i+1)}=\hat{s}_{ij}$, where $i=1,\ldots,n-2,~j=i+2,\ldots,n$. So the equation \eqref{sij} can be written more simply as follows
\begin{align}
\hat{s}_{i(i+1)}=\hat{s}_{i(i+2)}=\cdots =\hat{s}_{in}&=\frac{\alpha -(\hat{s}_{1i}^2+\hat{s}_{2i}+\cdots +\hat{s}_{(i-1)i}^2)}{\hat{s}_{ii}}=\frac{\alpha -(1-\hat{s}_{ii}^2)}{\hat{s}_{ii}}\nonumber\\&=\hat{s} _{ii}-\frac{1-\alpha}{\hat{s}_{ii}},
\end{align}
where this equation is used for computation of $i$th row.
\end{proof}
\begin{theorem}
Let $A\in\mathbb{R}^{n\times n}$ is nonsingular, with the Schur form $A=QTQ^T$ where $[T]_{ij}=t_{ij}$. Also for a pair of entries like $t_{ii}\neq t_{(i+1)(i+1)}$, there is an $\alpha\in(\frac{-1}{n-1},1)$ satisfies in $\frac{\hat{s}_{i(i+1)}}{\hat{s}_{(i+1)(i+1)}}=\frac{t_{i(i+1)}}{t_{(i+1)(i+1)}-t_{ii}}$, so that $\hat{s}_{i(i+1)},~\hat{s}_{(i+1)(i+1)}$ are followed from the upper triangular $\hat{S}\in \text{EM}^n_\alpha$ in \eqref{shat}. Furthermore the equality $T\hat{S}=\hat{S}\text{diag}(T)$, holds as noted in Proposition \ref{STiS}. Then the column vectors of $S=Q\hat{S}$ form the eigenvectors of $A$.
\end{theorem}
\begin{proof}
For two $t_{ii}\neq t_{(i+1)(i+1)}$ which satisfy in the hypothesis, we consider the entry $(i,i+1)$ in two sides of $T\hat{S}=\hat{S}\text{diag}(T)$. Now it can be checked that $A=Q\hat{S}\text{diag}(T)\hat{S}^{-1}Q^T$, then the result is yielded.
\end{proof}
For a $S\in \text{EM}^n_\alpha$, since $S^{-1}SS^{T}S= G_\alpha$, then $SS^T$ is similar to $G_\alpha$. From Lemma \ref{Leig}, $\sigma (SS^T)=\{1-\alpha ,1+(n-1)\alpha\}$, with the algebraic multiplicity $n-1$ at $1-\alpha$. Moreover, from Proposition \ref{Tinv},
\begin{equation}
(SS^T)^{-1}=S^{-T}S^{-1}=\beta(\beta^{-1/2}S^{-T} \beta^{-1/2}S^{-1})=\beta(\tilde{S} \tilde{S}^T), \label{rowe}
\end{equation}
where $\tilde{S}=\beta^{-1/2}S^{-T}\in \text{EM}^n_{\alpha^{\prime}}$. Now, assume that $A\in\mathbb{R}^{n\times n}$ is a symmetric matrix with two eigenvalues $\lambda_1,~\lambda_2=\lambda_1 +n(1-\lambda_1)$ with the algebraic multiplicity $n-1$ at $\lambda_1$, where $0<\lambda_1<1<\lambda_2$. Taking $\alpha=1-\lambda_1$ implies $\sigma (A)=\{1-\alpha ,1+(n-1)\alpha\}$. Therefore, for any $S\in EM^n_\alpha$, $A$ is orthogonally similar to $SS^T$. Then from the Schur form, $A=Q(SS^T)Q^T$, where $Q$ is an orthogonal matrix. Thus $A=(QS)(QS)^T=\tilde{S}\tilde{S}^T$, where $\tilde{S}\in \text{EM}^n_\alpha$. Now, the following Lemma is concluded.
\begin{lemma}
Let $A\in\mathbb{R}^{n\times n}$ is symmetric with two nonzero distinct eigenvalues $\lambda_1$, $\lambda_2$ where $\lambda_1\lambda_2 >0$ by the algebraic multiplicity $n-1$ at $\lambda_1$. Then there is a nonzero $r\in\mathbb{R}$ so that $A$ can be factorized as $rSS^T$ uniquely, where $S$ is an equiangular matrix.\label{LeiS}
\end{lemma}
\begin{proof}
There are two cases:
\begin{enumerate}
\item $\vert\lambda_1\vert<\vert\lambda_2\vert$\\\\
The values $\alpha=\tfrac{\lambda_2-\lambda_1}{\lambda_2-\lambda_1+n\lambda_1}\in(\frac{-1}{n-1},1)$ and $r=\tfrac{\lambda_2-\lambda_1+n\lambda_1}{n}$ are satisfy in the equations $r(1-\alpha)=\lambda_1$ and $r(1+(n-1)\alpha)=\lambda_2$. Now the desired eigenvalues $\lambda_1,\lambda_2$ are obtained. As noted before, $A$ can be factorized as $rSS^T$, where $S\in \text{EM}^n_\alpha$.
\item $\vert\lambda_2\vert<\vert\lambda_1\vert$\\\\
Since $A^{-1}$ has two eigenvalues $\lambda_1^{-1},\lambda_2^{-1}$, then from the previous case $A^{-1}=r^{\prime}\hat{S}\hat{S} ^T$, with the corresponding $r^{\prime},\alpha^\prime$. Then from \eqref{rowe}, $A=\dfrac{1}{r^\prime}(\hat{S}\hat{S}^T)^{-1}=\dfrac{\beta}{r^\prime}(SS^T)=rSS^T$, where $S\in \text{EM}^n_{\alpha^\prime}$ with $ \alpha^\prime =\frac{-\alpha}{1+(n-2)\alpha}$ and $r=\beta/r^{\prime}$.
\end{enumerate}
\end{proof}
\begin{theorem} 
(Generalization of the symmetric Schur form) Given a symmetric matrix $A\in\mathbb{R}^{n\times n}$ with at most $n-2$ zero eigenvalues and distinct nonzero eigenvalues. Then there is a $S\in \text{EM}^n_\alpha$, with $\alpha$ in a neighborhood of zero and a diagonal $D=\text{diag}(d_1,\ldots,d_n)$ so that $A=SDS^T$. Furthermore, if $\sigma(A)=\{\lambda_1,\ldots ,\lambda_n\}$, then $\sum d_i=\sum \lambda_i$. \label{SDST}
\end{theorem} 
\begin{proof}
Without loss of generality, assume that $A$ is nonsingular with distinct eigenvalues. Since for $A=\text{diag}(\lambda_1,\ldots ,\lambda_k,0,\ldots,0)$ if $\text{diag}(\lambda_1,\ldots ,\lambda_k)=S_kD_kS^T_k$, where $D_k=\text{diag}(d_1,\ldots ,d_k)$, then one can obtain the decomposition $A=SDS^T$ as follows
\begin{equation}
A=\left[\footnotesize{\begin{array}{c:c}
\lambda _1\qquad &~\\\ddots &0\\ \qquad\lambda _k&~\\ \hdashline ~&~\\0& 0 \\ \end{array}}\right]= \left[\footnotesize{\begin{array}{c:c}
~ &~\\S_k&~\\ ~&\tilde{S}_{n-k} \\ -------&~\\0 &~ \end{array}}\right] \left[\footnotesize{\begin{array}{c:c}
d_1\qquad &~\\\ddots &0\\ \qquad d_k&~\\ \hdashline ~&~\\0& 0 \\ \end{array}}\right] \left[\footnotesize{\begin{array}{c:c}
~ &~\\S_k&~\\ ~&\tilde{S}_{n-k} \\ -------&~\\0 &~ \end{array}}\right]^T,
\end{equation}
where $S_k$ is extended by deriving $\tilde{S}_{n-k}$ from the vectors $e_{k+1},\ldots ,e_n$ via Equiangular Algorithm.\\
Now, if there exists a diagonal matrix $D$ so that $\bar{S}_\alpha D\bar{S}_\alpha$ is orthogonally similar to $A$ as $A=P\bar{S}_\alpha D\bar{S}_\alpha P$, then the proof is coended: since $P\bar{S}_\alpha\in EM_\alpha^n$, then we set $S=P\bar{S}_\alpha$, so that $A=SDS^T$.
\\For computing $D$ we can write $\bar{S}_\alpha D\bar{S}_\alpha=\bar{S}_\alpha D{\bar{S}_\alpha}^2{\bar{S}_\alpha} ^{-1}=\bar{S}_\alpha DG_{\alpha}{\bar{S}_\alpha}^{-1}$, which indicates that $DG_{\alpha}$ must be similar to $\Lambda$. So the characteristic polynomials of them are the same. We discuss the characteristic polynomial of $DG_{\alpha}$ as follows
\begin{align}
\text{det}(xI-DG_{\alpha})&=\text{det} \footnotesize{\begin{bmatrix}x-d_{1}&-\alpha d_{1}&\cdots &-\alpha d_{1}\\-\alpha d_{2} &x-d_{2}&~&\vdots\\\vdots &~&\ddots &~\\-\alpha d_{n} &\cdots &~&x-d_{n}\end{bmatrix}} =(-1)^{n}\alpha ^{n}(\prod_{i=1}^{n}d_{i})~\text{det} \footnotesize{ \begin{bmatrix}\tfrac{x-d_{1}}{-\alpha d_{1}}&1&\cdots &1\\1&\tfrac{x-d_{2}}{-\alpha d_{2}}&~&\vdots\\\vdots &~&\ddots &1\\1&\cdots &1&\tfrac{x-d_{n}}{-\alpha d_{n}} \end{bmatrix}}\nonumber\\&=(-1)^{n}\alpha ^{n}(\prod_{i=1}^{n}d_{i})~\text{det}(\text{diag}(\tfrac{x-d_{1}}{-\alpha d_{1}}-1,~\tfrac{x-d_{2}}{-\alpha d_{2}}-1,\dots ,\tfrac{x-d_{n}}{-\alpha d_{n}}-1)+ee^T) \cdot \label{detGn1}
\end{align}
For the last determinant we use the Sherman-Morrison formula as follows
\begin{align}
\text{det}(xI-DG_{\alpha})=&(-1)^{n}\alpha ^{n}\prod_{i=1}^{n}d_{i}\left(\prod_{i=1}^{n}(\tfrac{x-d_{i}}{-\alpha d_{i}}-1)\right)\left(1+\sum_{i=1}^{n}\tfrac{-\alpha d_i}{x-(1-\alpha)d_i}\right)\nonumber\\=&(-1)^{n}\alpha ^{n}\prod_{i=1}^{n}d_{i}\left[\prod_{i\neq n}(\tfrac{x-d_{i}}{-\alpha d_{i}}-1)+\cdots +\prod_{i\neq 1}(\tfrac{x-d_{i}}{-\alpha d_{i}}-1)+\prod_{i=1}^n (\tfrac{x-d_{i}}{-\alpha d_{i}}-1)\right]\nonumber\\=& (-1)^{n}\alpha ^{n}\prod_{i=1}^{n}d_{i}~\frac{1}{(-1)^{n}\alpha ^{n}(\prod_{i=1}^{n}d_{i})}~[(-\alpha d_{n}\prod_{i\neq n}(x-d_{i}(1-\alpha))\nonumber\\&- \cdots-\alpha d_{1}\prod_{i\neq 1}(x-d_{i}(1-\alpha))+\prod_{i=1}^{n}(x-d_{i}(1-\alpha))]=\cdots
\end{align}
\begin{align}
=&x^{n}+(-\alpha +\alpha -1)(\sum_{i=1}^{n}d_{i})x^{n-1}+(-2\alpha (\alpha -1)+(\alpha -1)^2)(\sum_{1\leq i< j\leq n}d_{i}d_{j})x^{n-2}\nonumber\\&+\cdots +(-(n-1)\alpha (\alpha -1)^{n-2}+(\alpha -1)^{n-1})(\sum _{1\leq i_{j}\leq n}d_{i_{1}}\cdots d_{i_{n-1}})x+(-n\alpha (\alpha -1)^{n-1}+(\alpha -1)^n)d_{1}\cdots d_{n}\nonumber\\=&x^{n}- (\sum_{i=1}^{n} d_{i})x^{n-1}+(1-\alpha ^2)(\sum_{1\leq i< j\leq n}d_{i}d_{j}) x^{n-2}-\cdots -(\alpha -1)^{n-2} ((n-2)\alpha +1)(\sum_{1\leq i_{j}\leq n}d_{i_{1}}\cdots d_{i_{n-1}}) x\nonumber\\&-(\alpha -1)^{n-1}(1+(n-1)\alpha) d_{1}\cdots d_{n}\cdot
\end{align}
On the other hand, the characteristic polynomial of $\Lambda$ is illustrated as follows
\begin{align}
p(x)=&(x-\lambda_{1})(x-\lambda _{2})\cdots(x-\lambda_{n})=x^{n}-(\sum _{i=1}^{n}\lambda_{i})x^{n-1}+ (\sum_{1\leq i< j\leq n}\lambda_{i}\lambda_{j})x^{n-2} \nonumber\\&-\cdots +(-1)^{n-1}(\sum_{1\leq i_{j}\leq n}\lambda_{i_{1}}\cdots \lambda_{i_{n-1}})x +(-1)^{n}\lambda _{1}\cdots \lambda_{n}\cdot
\end{align}
Since det$(xI-DG_\alpha)=\text{det}(xI-\Lambda)$, then
\begin{align}
&\sum_{i=1}^{n}d_{i}=\sum_{i=1}^{n} \lambda_{i}=c_1 \Rightarrow \text{trace}(D)=\text{trace}(\Lambda),\nonumber \\& \sum_{1\leq i< j\leq n}d_{i}d_{j}=\frac{1}{1-\alpha ^2}\sum_{1\leq i< j\leq n}\lambda_{i}\lambda_{j}=c_2,\nonumber\\&\qquad \vdots\nonumber\\& \sum_{1\leq i_{j}\leq n}d_{i_{1}}\cdots d_{i_{n-1}}=\frac{1}{(1-\alpha)^{n-2}(1+(n-2)\alpha)}\sum_{1\leq i_{j}\leq n}\lambda_{i_{1}}\cdots \lambda_{i_{n-1}}=c_{n-1},\nonumber\\&d_{1}\cdots d_{n}=\frac{1}{(1-\alpha)^{n-1}(1+(n-1)\alpha)}\cdot\lambda_{1}\cdots \lambda_{n}=c_{n}\cdot \label{d_i}
\end{align}
Therefore $d_{1},~d_{2},\ldots,d_{n}$ are the roots of the following polynomial 
\begin{equation}
g(x)=x^{n}-c_{1}x^{n-1}+\cdots+(-1)^{n-1}c_{n-1} x+ (-1)^{n}c_{n}\cdot \label{g_n}
\end{equation}
Then the necessary condition for the implementation of this factorization is that the roots of $g(x)$ are all real because $SDS^T$ must be symmetric. Note that the scalar $\alpha$ in the coefficients of $g(x)$ can be considered as the perturbation in those of $p(x)$. Since the roots of $p(x)$ are distinct, then by the ``intermediate value theorem", there is a $\alpha$ in a neighborhood of zero for which the roots of $g(x)$ are all ``real" and possibly distinct. One can use an argument from the complex analysis \cite{Ahlfors}: the eigenvalues of $A$ are the continuous function of $A$, even though they are not differentiable.
\end{proof}
One can provide a counter example that shows the eigenvalues of $A$ are not distinct: let $A=\text{diag}(0,1,1)$ then $D=\text{diag}(0,1+\sqrt{ \tfrac{\alpha ^2}{1-\alpha ^2}}i,1- \sqrt{\tfrac{\alpha ^2}{1-\alpha ^2}}i)$ which is not real and symmetric. So $A$ wont be factorized as $SDS^T$. In general, if $A$ is a factor of identity matrix: let $A=rI=SDS^T$, then $rS^{-1}S^{-T}=rG_\alpha ^{-1}=D$ that satisfies $G_\alpha$ is diagonal, which is a contradiction. Therefore, in this case, the corresponding $g(x)$ in \eqref{g_n} has at least two nonreal roots. Note that the distinction of the eigenvalues of $A$ is the sufficient condition, but not necessary. For example, if $A$ is symmetric by two nonzero eigenvalues with the algebraic multiplicity $n-1$ at one of them, then by the Lemma \ref{LeiS}, the decomposition $A=rSS^T$ is possible and it suffices to let $D=rI$.
\begin{proposition}
For $\alpha\in(\frac{-1}{n-1},1)$ and nonzero $r$, the following polynomial has at least two nonreal roots.\label{root}
\begin{equation}
g_{n}(x)=x^n-nrx^{n-1}+\dfrac{\binom{n}{2}r^2}{1-\alpha ^2}x^{n-2}-\dfrac{\binom{n}{3}r^3}{(1-\alpha)^2(1+2\alpha)}x^{n-3} +\cdots +(-1)^n\dfrac{r^n}{(1-\alpha)^ {n-1}(1+(n-1)\alpha)}\cdot
\end{equation} \label{Pro}
\end{proposition}
\begin{proof}
Since the Theorem \ref{SDST} does not hold for $A=rI$, then there is no symmetric $D$ that satisfies $A=SDS^T$. From \eqref{d_i}, the coefficients of $g(x)$ are computed as follows
\begin{equation}
c_1=\sum _{i=1}^n r =nr,\quad c_2=\frac{1}{1-\alpha ^2}\sum _{1\leq i<j\leq n}r^2 =\dfrac{\binom{n}{2}r^2}{1-\alpha ^2}~,\quad\cdots\quad ,\quad c_n =\frac{r^n}{(1-\alpha)^{n-1}(1+(n-1)\alpha)}\cdot
\end{equation}
\end{proof}
\begin{example}
One can check the accuracy of Proposition \eqref{root} for the cases $n=2,3$: for $n=2$ we have $g_2 (x)=x^2 -2rx+\frac{r^2}{1-\alpha ^2}$, so the roots of $g_2$ are nonreal: $x_{1,2}=r\pm \frac{r\alpha}{\sqrt{1-\alpha ^2}}i$ and $g_2>0$. For $n=3$, $g_3(x)=x^3 -3rx^2 +\frac{3r^2}{1-\alpha^2}x-\frac{r^3}{(1-\alpha)^2 (1+2\alpha)}$, then $g_3^\prime(x)=3x^2-6rx +\frac{3r^2}{1-\alpha^2}$. Clearly, $g_3^\prime=3g_2>0$, so $g_3$ has one real root. In general, $g_n^\prime=ng_{n-1}$. By induction, $g_{n-1}$ has at most $n-3$ real roots and by intermediate value theorem $g_n$ has at most $n-2$ real roots. \qquad\qquad\qquad\qquad\qquad\qquad\qquad\qquad\qquad\qquad\qquad\qquad\qquad\qquad \qquad $\lozenge$
\end{example}
\begin{example}
For $\Lambda =\text{diag}(1,2,3)$, from Theorem \ref{SDST}, one can find a bound for $\alpha$ for which the factorization $\Lambda=SDS^T$ holds for some $S\in \text{EM}^n_\alpha$. So we can write $p(x)=x^3 -6x^2 +11x-6$ and $g(x)=x^3 -6x^2 +\frac{11}{1-\alpha^2}x-\frac{6}{(1-\alpha)^2(1+2\alpha)}$. Considering MATLAB $\mathtt{roots}$ function, the roots of $g$ are all real if $\alpha\leq 0.1843$, by roundoff.\qquad $\lozenge$
\end{example}
As noted before, Lemma \ref{LeiS} shows that Theorem \ref{SDST} also holds for a special class of symmetric matrices with two nonzero eigenvalues by the multiplicity $n-1$ at one of them. In these cases, the corresponding diagonal matrix $D$ is a factor of identity matrix. 
\begin{example}
For $\Lambda =\text{diag}(1,1,2)$, Lemma \ref{LeiS} implies that $r=4/3$ and $\alpha =1/4$ so that $D$ is $(4/3)I_3$. Also, from \eqref{s,t2} $s=0.9856$ and $t=0.1196$ by roundoff. Therefore, $\bar{S}_\alpha=0.9856G_{0.1213}$ and Theorem \eqref{SDST} results\\
$P=\left[\footnotesize{\begin{array}{ccc} ~~0.8059&-0.1310&0.5774\\-0.2895&~~ 0.7634& 0.5774\\-0.5164&-0.6325&0.5774 \end{array} }\right]$ and $S=P^T \bar{S}_\alpha=\left[\footnotesize{\begin{array}{ccc}~~0.6979&-0.2507&-0.4472\\-0.1134& ~~0.6612&-0.5477\\~~0.7071&~~0.7071& ~~0.7071\end{array} }\right]$. Then $\Lambda=SDS^T$. \qquad ~~$\lozenge$
\end{example}
\begin{example}
Let $r=1-\alpha$ and on the assumption that $d=\frac{\alpha}{1-\alpha}>\tfrac{-1}{n}$, the general term of $g_n$ is $c_k =\frac{\binom{n}{k}r^k}{(1-\alpha)^{k-1}(1+(k-1)\alpha)}= \frac{\binom{n}{k}(1-\alpha)}{1+(k-1)\alpha}=\frac{\binom{n}{k}}{1+dk}$. Therefore $g_n(x)= x^n-\sum_{k=1}^n \frac{\binom{n}{k}}{a_k}x^{n-k}$ has at least two nonreal roots, where $a_k$ is an arithmetic sequence with the initial term $a_1=1+d>1$ and the common difference $d$. $\lozenge$
\end{example}
\begin{example}
Let $r=1+(k-1)\alpha$ and on the assumption that $d=\frac{\alpha}{1-\alpha}>\tfrac{-1}{n}$, the general term of $g_n$ is $c_k=\frac{\binom{n}{k}r^k}{(1-\alpha)^{k-1}(1+(k-1)\alpha)}= \binom{n}{k}(\frac{1+(k-1)\alpha}{1-\alpha})^{k-1}=\binom{n}{k}(1+dk)^{k-1}$. Therefore, $g_n(x)= x^n-\sum_{k=1}^n \binom{n}{k} a_k^{n-1}x^{n-k}$ has at least two nonreal roots, where $a_k$ is an arithmetic sequence with the initial term $a_1=1+d>1$ and the common difference $d$.\qquad \qquad \qquad \qquad \qquad \qquad \qquad \qquad \qquad \qquad \qquad \qquad \qquad \qquad \qquad \qquad \qquad \qquad \qquad \qquad \qquad ~ $\lozenge$
\end{example}
Another counter example is the case that $A$ has an eigenvalue with multiplicity less than $n-1$, where $n$ is the number of nonzero eigenvalues. Without loss of generality, assume that $A=\text{diag}(\lambda_1,\lambda_2,\ldots,\lambda_n),~(n>3)$ is nonsingular, where $\lambda_1=\cdots=\lambda_i$ $(i<n-1)$. We can prove it in the following lemma.
\begin{lemma}
If the nonzero eigenvalues of $A\in\mathbb{R}^{n\times n}$ is $\lambda_1,\lambda_2,\ldots,\lambda_n,~(n>3)$ where the $n-k$ of $\lambda_i$'s are the same as $\lambda_{i_1}=\cdots=\lambda_{i_{n-k}}$ $(2\leq k\leq n-2)$. Then there are no $S\in \text{EM}^n_\alpha$ and real diagonal $D$ satisfies $A=SDS^T$. \label{naghz}
\end{lemma}
\begin{proof}
Without loss of generality, assume that $A=\text{diag}(\lambda_1,\ldots,\lambda_n)$ with the nonzero eigenvalues so that $\lambda_{k+1}=\cdots=\lambda_n =\lambda$. Let us consider for a contradiction and suppose that there are $S$ and $D$ satisfying the hypothesise of the problem, so $S^{-1}AS^{-T}=D$. From the equations \eqref{ninv} and \eqref{sipsjp} and its outcome in section \ref{equi3}, $S^{-1}$ has the equiangular rows with the norm $\sqrt{\beta}$ and the cosine of the angle $\alpha^\prime$. Then analogous to \eqref{rowe}, there exists $\tilde{S}\in \text{EM}^n_{\alpha^\prime}$ so that $S^{-1}=\sqrt{\beta}\tilde{S}^T$. Therefore, $\tilde{S}^T\text{diag}(\lambda_1,\ldots,\lambda_k,\lambda,\cdots,\lambda)\tilde{S}=\frac{1}{\beta}D$. The subtraction of this equality from the equation $\lambda\tilde{S}^T I_n\tilde{S}=\lambda G_{\alpha^\prime}$ is illustrated as follows
\begin{equation}
\begin{bmatrix}(\lambda _1-\lambda)\tilde{s}_{11}&\ldots &(\lambda _{k}-\lambda)\tilde{s}_{k1}\\ \vdots &\ddots &\vdots\\ (\lambda _1-\lambda)\tilde{s}_{1n} &\ldots &(\lambda _{k}-\lambda)\tilde{s}_{kn} \end{bmatrix}\cdot \begin{bmatrix}\tilde{s}_{11}&\ldots &\tilde{s}_{1n}\\ \vdots &\ddots &\vdots\\ \tilde{s}_{k1} &\ldots &\tilde{s}_{kn} \end{bmatrix}= \footnotesize{\begin{bmatrix}\tfrac{d_1}{\beta}-\lambda &-\lambda\alpha ^\prime &\cdots &-\lambda\alpha ^\prime\\-\lambda\alpha ^\prime &\tfrac{d_2}{\beta}-\lambda &~&\vdots\\\vdots &~&\ddots &~\\-\lambda\alpha ^\prime &\cdots &~ &\tfrac{d_n}{\beta}-\lambda  \end{bmatrix}}\cdot\label{STAS}
\end{equation}
The right hand matrix is like the mentioned matrix in \eqref{detGn1} and its rank must be at most $k$. On the other hand, if $\alpha$ converges to zero, then the diagonal entries of the above matrix tend to $d_i-\lambda$ and all off-diagonals tend to zero. One can consider a subsequence of the $\alpha$'s close to zero in such a way that for each $i$, $\frac{d_i}{\beta}-\lambda\neq 0$ and also $\vert (n-1)\alpha^\prime\lambda\vert<\vert \frac{d_i}{\beta}-\lambda\vert$. Then there exists an $\frac{-1}{n-1}<\alpha<1$ for which 
\begin{equation}
0<\vert (n-1)\alpha^\prime\lambda\vert<\vert \frac{d_i}{\beta}-\lambda\vert\cdot
\end{equation}
Now, by the Gerschgorin Theorem \cite{Horn,Meyer}, the eigenvalues of the mentioned matrix are nonzero. So its rank will be $n>k$ which is a contradiction.
\end{proof}
Now, the outcome of the mentioned counter example can be represented as a theorem in relation to the general form of the polynomials with nonreal roots.
\begin{theorem}
The real scalers $\lambda_1,\lambda_2,\ldots,\lambda_n$ and $\frac{-1}{n-1}<\alpha<1,~(n>3)$ are given so that there are two cases for $\lambda_i$'s: either all of them are equal or the $i$ of them are equal $(2\leq i\leq n-2)$. Then the following polynomial has at least two nonreal roots. $(n\geq 2)$
\begin{equation}
f_{n}(x)=x^n-(\sum _{i=1}^n\lambda _i)x^{n-1}+\dfrac{\sum _{1\leq i<j\leq n}\lambda _i\lambda _j}{(1-\alpha)(1+\alpha)}x^{n-2}+\cdots +(-1)^n\dfrac{\lambda _1\ldots\lambda _n}{(1-\alpha)^ {n-1}(1+(n-1)\alpha)}\cdot \label{fnx}
\end{equation}
Moreover, all real polynomials of degree $n$, which have the nonreal roots, can be illustrated as the above form. 
\end{theorem}
\begin{proof}
From the Proposition \ref{root} and Lemma \ref{naghz}, the first part is proven. For the next part regarding the Lemma \ref{LeiS} and Theorem \ref{SDST} when the $\lambda_i$'s are distinct or $n-1$ of them are equal, then for an $\alpha$ all $d_i$'s are real. So we can say that if the scalers $d_i$ with at least two of them are nonreal, being the roots of $f_n(x)$, then the $\lambda_i$'s must be the same as mentioned in the hypothesise of the theorem.
\end{proof}
%%%%%%%%%%%%%%% section 4 %%%%%%%%%%%%%%%%%%%%
\section{Doubly equiangular matrices}\label{doubequi}
In this section, we study the special case of equiangular matrices which have not only equiangular columns, but also equiangular rows. The mentioned matrix $\bar{S}_\alpha$, in section \ref{equi3}, is of this type, but we want to find their general form.
\begin{definition}
The matrix $\bar{S}$ is called ``doubly equiangular matrix" $(\text{DEM}^n_\alpha)$ if $\bar{S},\bar{S}^T \in \text{EM}^n_\alpha$ for nonzero $\alpha\in(\frac{-1}{n-1},1)$. \label{ddoub}
\end{definition} 
	Note that a doubly equiangular matrix $\bar{S}$ is normal because $\bar{S}\bar{S}^T=\bar{S}^T\bar{S}$. So $\bar{S}$ is orthogonally diagonalizable. From Schur form $\bar{S}=Q\Lambda Q^T$ so that $\Lambda$ is a blocked diagonal matrix with $1\times 1$ and $2\times 2$ blocks. We have $\bar{S}^T\bar{S}=G_\alpha =Q\Lambda^T\Lambda Q^T$, where $\Lambda_\alpha=\Lambda^T\Lambda$ is a diagonal matrix with some eigenvalues corresponding to the $1\times 1$ blocks of $\Lambda$. With regarding to the Lemma \ref{Leig}, $\Lambda_\alpha$ has an eigenvalue $1+(n-1)\alpha$ in a $1\times 1$ block with the corresponding eigenvector $e$. So its corresponding eigenpair in $\bar{S}$ is $(\sqrt{1+(n-1)\alpha},e)$. Since $e$ is an eigenvector of $\bar{S}$, then the row sums of $\bar{S}$ are the same. The same reasoning is true for $\bar{S}^T$. So $e$ is also an eigenvector of $\bar{S}^T$ which implies that the column sums of $\bar{S}$ are the same. Then $\bar{S}e=\bar{S}^Te=(1+(n-1)\alpha)^ {1/2}e$. It implies that the vector $e$ is both the right and left eigenvector of $\bar{S}$. For the case of $\alpha=0$, we present a special definition similar to the previous one.
\begin{definition}
The orthogonal matrix $\bar{Q}$ is called ``doubly orthogonal matrix" if $e$ is the eigenvector of $\bar{Q}$. Also, $\text{DOM}^n$ denotes the set of all doubly orthogonal matrices.
\end{definition}
In this case the condition of being a normal matrix is not sufficient for being $e$ the eigenvector of matrix $Q$. Since $G_0=I$, then any vector can be an eigenvector of $1$ so that the eigenvectors of $I$ and accordingly $Q$ are not restricted to $e$. Although for any orthogonal matrix $Q$, $Q^TQ=QQ^T$, based on definition, it may not necessarily be doubly orthogonal, unless the vector $e$ is its eigenvector.\\
Now, we want to derive a doubly equiangular (doubly orthogonal) matrix $\bar{S}$ $(\bar{Q})$ via an equiangular (orthogonal) matrix $S$ $(Q)$ so that its column sums vector has the same direction as $e$. $(\bar{S}e=\lambda e$ or $\bar{Q}e=\lambda e)$.
\begin{theorem}
Let $S\in \text{EM}^n_\alpha$, then one can obtain a doubly equiangular (orthogonal) matrix $\bar{S}\in \text{DEM}^n_\alpha ~(\text{DOM}^n)$ from $S$ as follows
\begin{equation}
\bar{S}=(I-\frac{2uu^T}{\Vert u\Vert ^2})S\label{dequ}
\end{equation}
where $u=(S-\sqrt{1+(n-1)\alpha}~ I)e$. \label{dthe}
\end{theorem}
\begin{proof}
It suffices to construct a Householder transformation like $H$ so that it transforms the column sums vector of $S$ to a vector which lies in the same direction as $e$. The subtraction of normalized column sums vector of $S$ and $\bar{S}$ is $\frac{1}{\sqrt{n(1+(n-1)\alpha)}}Se-\frac{1}{\sqrt{n}}e$ which must be in the same direction as $u$. Then $\bar{S}$ is obtained by multiplication of the Householder matrix $I-\frac{2uu^T}{\Vert u\Vert ^2}$ and $S$ from the left. It can be checked that $\bar{S}\bar{S}^T=G_\alpha$ and $\bar{S}e=\bar{S}^Te=(1+(n-1)\alpha)^ {1/2}e$.
\end{proof}
As noted in Example \ref{psquare}, the sufficient condition for producing a doubly equiangular matrix with nonnegative entries is that $\frac{n-2}{n-1}\leq\alpha<1$, so that the lower bound of $\alpha$ is held for $\sqrt{G_\alpha}$ with $\tfrac{1}{\sqrt{n}}<t\leq\tfrac{1}{\sqrt{n-1}}$. Actually, $\sqrt{G_\alpha}$ in this case is a symmetric matrix, but not the principal root of $G_\alpha$. The general form of $\sqrt{G_\alpha}$ is illustrated as follows
\begin{equation}
\sqrt{G_\alpha}=\left[\begin{array}{cccc}0&\frac{1}{\sqrt{n-1}}&\ldots &\frac{1}{\sqrt{n-1}}\\\frac{1}{\sqrt{n-1}}&0&\ldots &\frac{1}{\sqrt{n-1}}\\\vdots &\vdots &\ddots &\vdots\\\frac{1}{\sqrt{n-1}}&\frac{1}{\sqrt{n-1}}&\ldots &0\end{array}\right]~,\qquad\alpha =\frac{n-2}{n-1}\cdot
\end{equation}
With regard to the matrix $\bar{S}_\alpha=G_\alpha^{1/2}$ that has positive entries, we deduce the mentioned condition for $\alpha$ is not necessary. Actually, matrix $(1+(n-1)\alpha)^ {-1/2}\bar{S}$ where $\bar{S}\in \text{DEM}^n_\alpha$ is a doubly stochastic matrix with the mentioned condition and a quasi doubly stochastic matrix without it.\\
One can provide an algorithm named DEA which is followed from the Theorem \ref{dthe} to derive a matrix $S\in \text{DEM}^n_\alpha~(Q\in \text{DOM}^n)$ from the decomposition SR (QR) of a nonsingular matrix $A$.
\begin{algorithm} 
\caption{.~This algorithm produces a doubly equiangular (orthogonal) matrix $\bar{S}~(\bar{Q})$ with $\alpha\in(\frac{-1}{n-1},1)$ from a nonsingular matrix $A$ of size $n$.}
%\begin{LTR}
\label{algo:dequ}
\begin{algorithmic}[1]
\STATE $A=SR~;$~\% SR (QR) decomposition with $\alpha\in(\frac{-1}{n-1},1)$
\STATE $u=(S-\sqrt{1+(n-1)\alpha}~I)e~;$
\STATE $\bar{S}=(I-\frac{2uu^T}{\Vert u\Vert ^2})S~;$
\end{algorithmic}
%\end{LTR}
\end{algorithm}
\begin{example}
For $A=\left[\footnotesize{\begin{array}{cccc}1&1/2 & 1/3 &1/4\\1/2 & 1/3 &1/4&1/5\\1/3 &1/4&1/5&1/6\\1/4&1/5&1/6&1/7 \end{array}}\right]$, by applying DEA on $A$, the rounded $\bar{S}$ with $\alpha=2/3$ is obtained as follows
\begin{equation}
\bar{S}=\left[\footnotesize{\begin{array}{cccc}0.8517&0.3048& 0.3774&0.1981\\ 0.3976&0.3942&0.1205&0.8198\\ 0.2399&0.1863&0.8189&0.4869\\0.2429 &0.8468&0.4152&0.2273 \end{array}}\right]~,
\end{equation}
where $\bar{S}\bar{S}^T=\bar{S}^T\bar{S}= G_{2/3}$ and $\bar{S}e=\bar{S}^T e=\sqrt{3}e$. \qquad\qquad\qquad\qquad\qquad\qquad\qquad\qquad\qquad\qquad\qquad\qquad\qquad $\lozenge$
\end{example}
\begin{example}
The orthogonal matrix $Q=\left[\footnotesize{\begin{array}{ccc}~~3/7&-2/7&~~6/7\\~~6/7&~~3/7 &-2/7\\-2/7&~~6/7&~~3/7\end{array}} \right]$, in Example \ref{ExQR}, is a cyclic doubly orthogonal matrix so that $Qe=Q^Te=e$. \qquad\qquad\qquad\qquad\qquad\qquad\qquad\qquad\qquad\qquad\qquad\qquad\qquad\qquad \qquad\qquad $\lozenge$
\end{example}
\begin{example}
The matrix $Q=\left[\footnotesize{\begin{array}{cccc}-1/2&~1/2&~1/2&~1/2\\-1/2&~1/2 &-1/2&-1/2\\-1/2&-1/2&~1/2 &-1/2\\-1/2&-1/2&-1/2&~1/2 \end{array}} \right]$ is an orthogonal matrix. From DEA, $Q$ is transformed to the following doubly orthogonal matrix
\begin{equation}
\bar{Q}=\left[\footnotesize{\begin{array}{cccc}-1/2&1/2&~1/2&~1/2\\~1/2&~5/6& -1/6&-1/6\\~1/2&-1/6&~5/6 &-1/6\\~1/2&-1/6&-1/6&~5/6 \end{array}} \right],
\end{equation}
so that $\bar{Q}e=\bar{Q}^Te=e$. \qquad\qquad\qquad\qquad\qquad\qquad\qquad\qquad\qquad\qquad\qquad\qquad\qquad\qquad \qquad\qquad\qquad\qquad\quad ~$\lozenge$
\end{example}
From the later discussions and examples, one can conclude the doubly orthogonal matrices are the quasi doubly stochastic matrices.
\begin{proposition}
The Gram matrix $G_\alpha$ is commutable with the family of all $\bar{S}\in \text{DEM}^n_{\alpha^\prime}$. \label{comm}
\end{proposition}
\begin{proof}
The proof stems from the fact that the vector $e$ is both the right and left eigenvector of $\bar{S}$.
\end{proof}
\begin{proposition}
Let $\bar{S}_1\in \text{DEM}^n_{\alpha_1}$ and $\bar{S}_2\in \text{DEM}^n_{\alpha_2}$, with $\alpha_1,\alpha_2\in(\frac{-1}{n-1},1)$. Then $\bar{S}_1\bar{S}_2$ is a factor of a doubly equiangular matrix, so it is normal.
\end{proposition}
\begin{proof}
From Proposition \ref{comm}, we have
\begin{equation}
\bar{S}_1\bar{S}_2(\bar{S}_1\bar{S}_2)^T =\bar{S}_1\bar{S}_2 \bar{S}_2^T\bar{S}_1^T=\bar{S}_1 G_{\alpha_2}\bar{S}_1^T=G_{\alpha_2} \bar{S}_1\bar{S}_1^T= G_{\alpha_2}G_{\alpha_1}=cG_{\alpha ^\prime},
\end{equation}
where $c=1+(n-1)\alpha_1 \alpha _2$ and $\alpha^\prime=\tfrac{\alpha_1 +\alpha_2 +(n-2)\alpha_1 \alpha_2}{c}$. Proving $(\bar{S}_1\bar{S}_2)^T\bar{S}_1\bar{S} _2=cG_{\alpha^\prime}$ is similar the last one.
\end{proof}
In the category of normal matrices, there are symmetric (hermitian), skew-symmetric (skew-hermitian) and orthogonal (unitary) matrices. Now, we can add the new type of normal matrices ``doubly equiangular (orthogonal) matrices" to that category. However, there exist normal matrices that are not included in any of them.
\section{Conclusion}
In this paper, we introduce an algorithm that produces an equiangular matrix. It then shows some of its results in the matrix factorization.
%\section*{References}


\begin{thebibliography}{99}
\bibitem{Ahlfors} L. Ahlfors, {\em Complex Analysis}, McGraw-Hill, New York, 1966.
\bibitem{Bjorck1} A. Bj\"orck, {\em Numerical Methods for Least Squares Problems}, SIAM, Philadelphia, PA, 1996.
\bibitem{Bjorck2} A. Bj\"orck, {\em Solving linear least squares problems by Gram-Schmidt orthogonalization}, BIT {\bf 7}, 1-21, 1967.
\bibitem{Bodmann} B. Bodmann and V. Paulsen, {\em Frames, graphs and erasures}, Linear Algebra. Appl, {\bf 404}, 118-146, 2005.
\bibitem{Casazza} P. Casazza and J. Kovacevic, {\em Equal-norm tight frames with erasures}, Adv. Comp. Math, {\bf 18}, 387-430, 2003.
\bibitem{Casazza2} P. Casazza and G. Kutyniok, {\em A generalization of Gram-Schmidt orthogonalization generating all Parseval frames}, Adv. Comput. Math,  {\bf 27}, 65-78, 2007.
\bibitem{Duffin} R. J. Duffin and A. C. Schaeffer, {\em A class of nonharmonic Fourier series}, Trans. Amer. Math. Soc, {\bf 72}, 341-366, 1952.
\bibitem{Godsil} C. D. Godsil, and G. Royle, {\em Algebraic graph theory}, Springer-Verlag, New York, 2001.
\bibitem{Golub} G. H. Golub and C. Van Loan, {\em Matrix Computations} (3rd ed.), Johns Hopkins, ISBN 978-0-8018-5414-9, 1996.
\bibitem{Health} R.W. Heath, T. Strohmer and A. J. Paulraj, {\em On quasi-orthogonal signatures for CDMA}, IEEE Trans. Information Theory 52 No, {\bf 3}, 1217-1226, 2006.
\bibitem{Higham1} N. J. Higham. {\em Computing the polar decomposition with applications}, SIAM J. Sci. Stat. Comput. Philadelphia, PA, USA, {\bf 7}:4, 1160–1174, 1986.
\bibitem{Higham} N. J. Higham. {\em Functions of Matrices: Theory and Computation}, Society for Industrial and Applied Mathematics, Philadelphia, PA, USA, 2008.
\bibitem{Holmes} R. B. Holmes and V. I. Paulsen, {\em Optimal frames for erasures}, Linear Algebra. Appl, {\bf 377}, 31-51, 2004.
\bibitem{Horn} R. A. Horn and C. R. Johnson, Matrix Analysis, Cambridge University Press, Cambridge, 1985.
\bibitem{Lanckriet}G. R. G. Lanckriet, N. Cristianini, P. Bartlett, L. E. Ghaoui, M. I. Jordan, Learning the kernel matrix with semidefinite programming, Journal of Machine Learning Research, {\bf 5}, 27–72, p. 29, 2004.
\bibitem{Lemmens} P. Lemmens and J. Seidel, {\em Equiangular lines}, J. Algebra, {\bf 24}, 494-512, 1973.
\bibitem{Meyer} C.D. Meyer, {\em Matrix Analysis and Applied Linear Algebra}. SIAM 2001.
\bibitem{Schmidt} E. Schmidt, {\em \"Uber die Aufl\"osung linearer Gleichungen mit unendlich vielen Unbekannten}, Rend. Circ. Math. Palermo. Ser. 1, {\bf 25}, 53-77, 1908.
\bibitem{Soliverez} E. Soliv\' erez, E. Gagliano,  {\em Orthonormalization on the plane: a geometric approach}, Mex. J. Phys, {\bf 31}, no. 4, 743-758, 1985.
\bibitem{Strohmer} T. Strohmer and R. W. Heath, {\em Grassmannian frames with applications to coding and communication}, Appl. Comp. Harmonic Anal, 14 No. {\bf 3}, 2003, 257-275.
\bibitem{Terefethen} L. N. Trefethen, and D. Bau, {\em Numerical Linear Algebra III, Philadelphia}, PA: SIAM, 1997.
\bibitem{Todhunter} I. Todhunter, {\em Spherical Trigonometry for the use of college and schools, with numerous examples}, fifth edition, Macmillan and Co., London, 1886; an excellent source on spherical geometry may be found at www.gutenberg.org/ebooks/19770Cached Nov 12, 2006.
\end{thebibliography}
\end{document}